\documentclass[final,3p,times]{elsarticle}

\usepackage{framed,multirow}

\usepackage{amssymb}
\usepackage{latexsym}

\usepackage{url}
\usepackage{xcolor}
\definecolor{newcolor}{rgb}{.8,.349,.1}

\usepackage{amsmath}
\usepackage{amsfonts}
\usepackage{epsfig}
\usepackage{cancel}
\usepackage{lineno}
\usepackage{subfigure}
\usepackage{enumitem}
\usepackage{tikz}
\usepackage{algorithm}
\usepackage{algorithmic}
\usepackage{svg}

\usepackage{hyperref}
\usepackage{cleveref}
\crefname{equation}{}{}
\crefname{figure}{Fig.}{Figs.}
\crefname{appendix}{}{}
\crefname{table}{Tab.}{Tabs.}
\Crefname{ALC@unique}{Line}{Lines} 

\def\ol{\overline}

\def\Psib{\boldsymbol{\Psi}}
\def\omegab{\boldsymbol{\omega}}

\newcommand{\gray}[1]{\textcolor{gray!75}{#1}}

\DeclareMathAlphabet\mathbfcal{OMS}{cmsy}{b}{n}
\newtheorem{lemma}{Lemma}[section]

\newtheorem{theorem}{Theorem}[section]
\newtheorem{definition}{Definition}[section]
\newdefinition{remark}{Remark}[section]
\newproof{proof}{Proof}

\DeclareMathOperator{\diag}{diag}

\newcommand{\avg}[1]{\langle #1 \rangle} 
\newcommand{\jmp}[1]{[\![#1]\!]} 

\begin{document}


\begin{frontmatter}

\title{Maximum principle preserving time implicit DGSEM for linear scalar hyperbolic conservation laws}

\author[rvt1]{Riccardo Milani}
\ead{riccardo.milani@onera.fr}
\address[rvt1]{DAAA, ONERA, Universit\'e Paris Saclay, F-92322 Ch\^atillon, France}

\author[rvt1]{Florent Renac\corref{cor2}}
\ead{florent.renac@onera.fr}
\cortext[cor2]{Corresponding author. Tel.: +33 1 46 73 37 44; fax.: +33 1 46 73 41 66.}

\author[rvt1]{Jean Ruel}
\ead{jean.ruel@ens-paris-saclay.fr}


\begin{abstract}
We investigate the properties of the high-order discontinuous Galerkin spectral element method (DGSEM) with implicit backward-Euler time stepping for the approximation of hyperbolic linear scalar conservation equation in multiple space dimensions. We first prove that the DGSEM scheme in one space dimension preserves a maximum principle for the cell-averaged solution when the time step is large enough. This property however no longer holds in multiple space dimensions and we propose to use the flux-corrected transport limiting \cite{BORIS_Book_FCT_73} based on a low-order approximation using graph viscosity to impose a maximum principle on the cell-averaged solution. These results allow to use a linear scaling limiter \cite{zhang_shu_10a} in order to impose a maximum principle at nodal values within elements. Then, we investigate the inversion of the linear systems resulting from the time implicit discretization at each time step. We prove that the diagonal blocks are invertible and provide efficient algorithms for their inversion. Numerical experiments in one and two space dimensions are presented to illustrate the conclusions of the present analyses.
\end{abstract}


\end{frontmatter}


%
%
\section{Introduction}

We are here interested in the accurate and robust approximation of the following problem with an hyperbolic scalar linear conservation law in $d\geq1$ space dimensions:

\begin{subequations}\label{eq:hyp_cons_laws}
\begin{align}
 \partial_tu + \nabla\cdot{\bf f}(u) &= 0, \quad \text{in }\Omega\times(0,\infty), \label{eq:hyp_cons_laws-a} \\
 u(\cdot,0) &= u_{0}(\cdot),\quad\text{in }\Omega, \label{eq:hyp_cons_laws-b} 
\end{align}
\end{subequations}

\noindent with $\Omega\subset\mathbb{R}^d$, appropriate boundary conditions on $\partial\Omega$ (e.g., inflow or outflow conditions, periodic conditions), a smooth flux ${\bf f}\in{\cal C}^1(\mathbb{R},\mathbb{R}^d)$ and $u_{0}$ in $L^\infty(\mathbb{R}^d,\mathbb{R})$. We here consider linear fluxes with constant coefficients ${\bf f}(u)={\bf c}u$ with a given ${\bf c}$ in $\mathbb{R}^d$. Without loss of generality, we assume ${\bf c}$ in $\mathbb{R}_+^d$, a negative component being handled by reverting the corresponding space direction.

Problem \cref{eq:hyp_cons_laws} has to be understood in the sense of distributions where we look for weak solutions that are piecewise ${\cal C}^1$ solutions. Introducing the square entropy $\eta(u)=\tfrac{u^2}{2}$ and associated entropy flux ${\bf q}(u)={\bf c}\tfrac{u^2}{2}$ pair, solutions to \cref{eq:hyp_cons_laws-a} also satisfy

\begin{equation}\label{eq:PDE_entropy_ineq}
 \partial_t\eta(u) + \nabla\cdot{\bf q}(u) \leq 0, \quad \mbox{in }\Omega\times(0,\infty),
\end{equation}

\noindent in the sense of distributions. For compactly supported solutions, this brings uniqueness and $L^2$ stability:

\begin{equation*}
 \|u\|_{L^2(\Omega,\mathbb{R})} \leq \|u_0\|_{L^2(\Omega,\mathbb{R})} \quad \forall t\geq0.
\end{equation*}

Solutions to \cref{eq:hyp_cons_laws} also satisfy a maximum principle:

\begin{equation}\label{eq:PDE_max_principle}
 m \leq u_0(x) \leq M \text{ in }\Omega \quad\Rightarrow\quad
 m \leq u(x,t) \leq M \text{ in }\Omega\times(0,\infty),
\end{equation}

\noindent almost everywhere, which brings $L^\infty$ stability.


We are here interested in the approximation of \cref{eq:hyp_cons_laws} with a high-order space discretization that satisfies the above properties at the discrete level. We here consider the discontinuous Galerkin spectral element method (DGSEM) based on collocation between interpolation and quadrature points \cite{kopriva_gassner10} and tensor products of one-dimensional (1D) function bases and quadrature rules. The collocation property of the DGSEM associated to tensor-product evaluations and sum factorization drastically reduces the number of operations in the operators implementing the discretization and makes the DGSEM computationally efficient. Moreover, using diagonal norm summation-by-parts (SBP) operators and the entropy conservative numerical fluxes from Tadmor \cite{tadmor87}, semi-discrete entropy conservative finite-difference and spectral collocation schemes have been derived in \cite{fisher_carpenter_13,carpenter_etal14} and applied to a large variety of nonlinear conservation laws \cite{winters_etal_16,bohm_etal_18,despres98,renac17a,renac17b,renac2020entropy,winters2016affordable,Peyvan_etal_reactive_DG_23,zhang_shu_10a}, nonconservative hyperbolic systems and balance laws \cite{liu_etal_ES_DGSEM_MHD_18,renac19,coquel_etal_DGSEM_BN_21,ARR_HLLCgam_22,winters_etal_17,Waruszewski_etal_noncons_DG_22}, among others.

Most of the time, these schemes are analyzed in semi-discrete form for which the time derivative is not discretized, or when coupled with explicit in time discretizations. Time explicit integration may however become prohibitive for long time simulations or when looking for stationary solutions due to the strong CFL restriction on the time step which gets smaller as the approximation order of the scheme increases to ensure either linear stability \cite{gassner_kopriva_disp_diss_11,atkins_shu_98,krivodonova_qui_CFL_DG_13}, or positivity of the approximate solution \cite{zhang_shu_10a,zhang2010positivity}. The DGSEM also presents attractive features for implicit time stepping. First, the collocation property reduces the connectivity between degrees of freedom (DOFs) which makes the DGSEM well suited due to a reduced number of entries in the Jacobian matrix of the space residuals. This property has been used in \cite{ruedaramirez_etal_static_DGSEM_21} to rewrite the time implicit discretization of the compressible Navier-Stokes equations as a Schur complement problem at the cell level that is then efficiently solved using static condensation. Then, tensor-product bases and quadratures have motivated the derivation of tensor-product based approximations of the diagonal blocks of the Jacobian matrix by Kronecker products \cite{VanLoan2016,vanLoanKron00} of 1D operators using singular value decomposition of a shuffled matrix \cite{pazner_persson_percon_DG_18}, or a least squares alternatively in each space direction \cite{diosady_murman_precond_19}.


We here consider and analyze a DGSEM discretization in space associated with a first-order backward-Euler time integration which allows to circumvent the CFL condition for linear stability and makes it well adapted for approximating stationary solutions or solutions containing low characteristic frequency scales. It is however of strong importance to also evaluate to what extent other properties of the exact solution are also satisfied at the discrete level. Positivity preservation for instance is an essential property that is required to prevent the computations from crashing due to floating exceptions during the simulation of many hyperbolic systems. Little is known about the properties of time implicit DGSEM schemes, apart from the entropy stability which holds providing the semi-discrete scheme is entropy stable due to the dissipative character of the backward-Euler time integration. An analysis of a time implicit discontinuous Galerkin (DG) method with Legendre basis functions for the discretization of a 1D linear scalar hyperbolic equation has been performed in \cite{QinShu_impt_positive_DG_18} and showed that a lower bound on the time step is required for the cell-averaged solution to satisfy a maximum principle at the discrete level. A linear scaling limiter of the DG solution around its cell-average \cite{zhang_shu_10a} is then used to obtain a maximum principle preserving scheme. Numerical experiments with linear and also nonlinear hyperbolic scalar equations and systems support the conclusion of this analysis. The theoretical proof of this lower bound uses the truncated expansion of the Dirac delta function in Legendre series that is then used as a test function in the DG scheme to prove that the Jacobian matrix of the cell-averaged discrete scheme is an M-matrix. It is however difficult to use this trick in the DGSEM scheme that uses lower-order quadrature rules and whose form is directly linked to the particular choice of Lagrange interpolation polynomials as test functions. Unfortunately, this discrete preservation of the maximum principle or positivity no longer holds in general in multiple space dimensions even on Cartesian grids and solutions with negative cell-average in some cell can be generated \cite{ling_etal_pos_impl_DGM_18}. In the case of linear hyperbolic equations and radiative transfer equations, Ling et al. \cite{ling_etal_pos_impl_DGM_18} showed that it is possible to impose positivity of the solution providing the approximation polynomial space is enriched with additional functions. The use of reduced order quadrature rules and suitable test functions were proposed in \cite{Xu_Shu_PP_DG_2022} to define a conservative scheme that preserves positivity in the case of stationary linear hyperbolic conservation laws, The work in \cite{Xu_Shu_cons_PP_2023} proposes limiters that allow to ensure positivity of stationary solutions of the radiative transfer equations, while keeping a particular local conservation property for stationary conservation laws. These modifications seem difficult to be directly applied to the DGSEM without loosing the collocation which is essential for the efficiency of the method. A limiter for time implicit DG schemes for nonlinear scalar equations has been proposed in \cite{vandervegt_lim_impl_DG_19} by reformulating the discrete problem as a constrained optimization problem and introducing Lagrange multipliers associated to the constraints. This however results in a nonlinear and nonsmooth algebraic system of equations that requires an adapted Newton method for its resolution.

In the present work, we propose an analysis of the DGSEM scheme with backward-Euler time stepping for linear hyperbolic equations on Cartesian grids. We first analyze the discrete preservation of the maximum principle property and show that it holds for the cell-averaged solution in one space dimension for sufficiently large time steps. This result is similar to the one obtained in \cite{QinShu_impt_positive_DG_18} for a modal DG scheme with Legendre polynomials, though the conditions on the time step are different. The proof relies on the nilpotent property of the discrete derivative matrix evaluating the derivatives of the Lagrange interpolation polynomials at quadrature points. This property allows to easily invert the mass and stiffness matrices and derive a scheme for the cell-averaged scheme, thus allowing to derive conditions for the associated Jacobian matrix to be an M-matrix. The DOFs are then limited with the linear scaling limiter from \cite{zhang_shu_10a} to impose a maximum principle to the whole solution. Unfortunately, this property no longer holds in multiple space dimensions similarly to the modal DG scheme \cite{ling_etal_pos_impl_DGM_18}. We thus follow \cite{Guermond_IDP_NS_2021,ern_guermond_IDP_DIRK_23} that propose to use the flux-corrected transport (FCT) limiter \cite{BORIS_Book_FCT_73,zalesak1979fully} combining a low-order and maximum principle preserving scheme with the high-order DGSEM. The low-order scheme is obtained by adding graph viscosity \cite{guermond_popov_GV_16,Guermond_etal_IDP_conv_lim_18,PAZNER_idg_DGSEM20211} to the DGSEM scheme. The FCT limiter is here designed to preserve a maximum principle for the cell-averaged solution, not for all the DOFs. This aspect is essential to reduce the effect of the limiter when the solution is smooth. In particular, the numerical experiments highlight a strong improvement of the accuracy of the limited scheme that would be otherwise affected when limiting the DOFs as already observed in the literature \cite{Guermond_etal_IDP_conv_lim_18,Guermond_etal_IDP_conv_lim_19}. Here again, the linear scaling limiter is applied after the FCT limiter to ensure the maximum principle on the whole solution.

We also analyze the inversion of the linear system resulting from the time implicit discretization to be solved at each time step. The linear system is large, non symmetric, sparse with a sparsity pattern containing dense diagonal and sparse off-diagonal blocks of size the number of DOFs per cell. Efficient inversion could be achieved through the use of block sparse direct or iterative linear solvers. Many algorithms require the inversion of the diagonal blocks as in block-preconditionned Krylov solvers \cite{pazner_persson_percon_DG_18,persson_peraire_GMRES_DG_08,crivelini_bassi_mat_free_DG_11}, block relaxation schemes \cite{saad_lin_solvers_book}, etc. We here prove that the diagonal blocks are invertible and propose efficient algorithms for their inversion\footnote{A repository of the algorithms for block inversion is available at \url{https://github.com/rueljean/fast_DGSEM_block_inversion}. Consult \cref{app:fast_inversion_diag_blocks} for a description of the repository.}. We again use the nilpotency of the discrete derivative matrix to inverse the diagonal blocks of the 1D scheme. We use the inversion of the 1D diagonal blocks as building blocks for the inversion of diagonal blocks in multiple space dimensions thanks to the tensor product structure of the discretization operators. 


The paper is organized as follows. \Cref{sec:DGSEM_discr} introduces some properties of the DGSEM function space associated to Gauss-Lobatto quadrature rules. The 1D DGSEM is introduced and analyzed in \cref{sec:DGSEM_1D_fully_discr}, while \cref{sec:DGSEM_2D_fully_discr} focuses on the DGSEM in two space dimensions (see \cref{app:3D_DGSEM_scheme} for a summary of the results in three space dimensions). The results are assessed by numerical experiments in one and two space dimensions in \cref{sec:num_xp} and concluding remarks about this work are given in \cref{sec:conclusions}.


%
%
\section{The DGSEM discretization in space}\label{sec:DGSEM_discr}

\subsection{The DGSEM function space}

The DGSEM discretization consists in defining a discrete weak formulation of problem \cref{eq:hyp_cons_laws}. The space domain $\Omega$ is first discretized with a Cartesian grid $\Omega_h\subset\mathbb{R}^d$ with elements $\kappa$ labeled as $\kappa_i = [x_{i-1/2}, x_{i+1/2}]$ of size $\Delta x_i=x_{i+1/2}-x_{i-1/2}>0$, $1\leq i\leq N_x$, for $d=1$ (see \cref{fig:stencil_1D_DGSEM}); $\kappa_{ij} = [x_{i-1/2}, x_{i+1/2}]\times[y_{j-1/2}, y_{j+1/2}]$ of size $\Delta x_i\Delta y_j=(x_{i+1/2}-x_{i-1/2})(y_{j+1/2}-y_{j-1/2})>0$, $1\leq i\leq N_x$, $1\leq j\leq N_y$, for $d=2$ (see \cref{fig:stencil_2D_DGSEM}), etc. We also set $h\coloneqq\min_{\kappa\in\Omega_h}|\kappa|^{\frac{1}{d}}$.

The approximate solution to \cref{eq:hyp_cons_laws} is sought under the form (with some abuse in the notation for the indices and exponents that will be clarified below)

\begin{equation}\label{eq:DGSEM_num_sol}
 u_h({\bf x},t)=\sum_{k=1}^{N_p}\phi_\kappa^{k}({\bf x})U_\kappa^k(t) \quad \forall {\bf x}\in\kappa,\, \kappa\in\Omega_h,\, \forall t\geq0,
\end{equation}

\noindent where $(U_\kappa^k)_{1\leq k\leq N_p}$ are the DOFs in the element $\kappa$. The subset $(\phi_\kappa^{k})_{1\leq k\leq N_p}$ constitutes a basis of ${\cal V}_h^p$ restricted onto the element $\kappa$ and $N_p=(p+1)^d$ is its dimension. We use tensor product in each space direction of Lagrange interpolation polynomials $(\ell_k)_{0\leq k\leq p}$ associated to the Gauss-Lobatto quadrature nodes over $I=[-1,1]$, $\xi_0=-1<\xi_1<\dots<\xi_p=1$:

\begin{equation}\label{eq:def_Lag_polynom}
 \ell_k(\xi)=\prod_{l=0,l\neq k}^p \frac{\xi-\xi_l}{\xi_k-\xi_l}, \quad 0\leq k \leq p,
\end{equation}

\noindent which satisfy

\begin{equation}\label{eq:cardinalty_Lag_polynom}
 \ell_k(\xi_l)=\delta_{kl}, \quad 0\leq k,l \leq p,
\end{equation}

\noindent with $\delta_{kl}$ the Kronecker delta.

The basis functions are thus defined for $d=1$ by $\phi_i^{k}(x)=\ell_k\big(\tfrac{2}{\Delta x_i}(x-x_{i-\frac{1}{2}})-1\big)$ and for $d=2$ by $\phi_{ij}^{kl}({\bf x})=\ell_k\big(\tfrac{2}{\Delta x_i}(x-x_{i-\frac{1}{2}})-1\big)\ell_l\big(\tfrac{2}{\Delta y_j}(y-y_{j-\frac{1}{2}})-1\big)$, and so on.

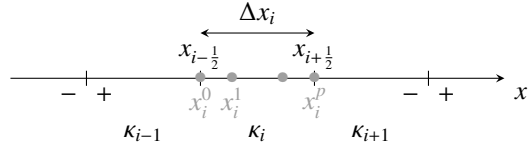
\begin{figure}
\begin{center}
\begin{tikzpicture}[scale=1]
\draw (3.25,0.6) node[above] {$\Delta x_i$};
\draw [>=stealth,<->] (2.5,0.6) -- (4.0,0.6) ;
\draw [>=stealth,->] (0,0) -- (6.5,0) ;
\draw (6.5,0) node[below right] {$x$};
\draw (1,-0.1) -- (1,0.1);
\draw (2.5,-0.1) -- (2.5,0.1);
\draw (4,-0.1) -- (4,0.1);
\draw (5.5,-0.1) -- (5.5,0.1);
\draw (2.5,0)   node {$\gray\bullet$};
\draw (2.915,0) node {$\gray\bullet$};
\draw (3.585,0) node {$\gray\bullet$};
\draw (4.0,0)   node {$\gray\bullet$};
\draw (2.5,0)   node[below] {$\gray{x_i^0}$};
\draw (2.5,0)   node[above] {$x_{i-\frac{1}{2}}$};
\draw (2.915,0) node[below] {$\gray{x_i^1}$};
\draw (4.0,0)   node[below] {$\gray{x_i^p}$};
\draw (4.0,0)   node[above] {$x_{i+\frac{1}{2}}$};
\draw (5.5,0)   node[below left] {$-$};
\draw (5.5,0)   node[below right] {$+$};
\draw (1,0)   node[below left] {$-$};
\draw (1,0)   node[below right] {$+$};

\draw (1.75,-0.5) node[below] {$\kappa_{i-1}$};
\draw (3.25,-0.5) node[below] {$\kappa_{i}$};,
\draw (4.75,-0.5) node[below] {$\kappa_{i+1}$};
\end{tikzpicture}
\caption{Mesh for $d=1$ with positions of quadrature points in $\kappa_i$, $x_i^k=\tfrac{1+\xi_k}{2}\Delta x_i+x_{i-\frac{1}{2}}$, for $p=3$.}
\label{fig:stencil_1D_DGSEM}
\end{center}
\end{figure}

The DGSEM uses Gauss-Lobatto quadrature rules to approximate the integrals over elements:

\begin{equation}\label{eq:GaussLobatto_quad}
 \int_{-1}^1 f(\xi)d\xi \simeq \sum_{k=0}^p \omega_k f(\xi_k),
\end{equation}

\noindent with $\omega_k>0$ and $\sum_{k=0}^p\omega_k=\int_{-1}^1ds=2$, the weights and $\xi_k$ the nodes over $I$ of the quadrature rule.

\begin{figure}
\begin{center}
\begin{tikzpicture}[scale=1]
\draw [>=stealth,->,thick] (-0.1,0) -- (6.5,0) ;
\draw (6.5,0) node[below] {$x$};
\draw [>=stealth,->,thick] (0,-0.1) -- (0,3.5) ;
\draw (0,3.5) node[left] {$y$};

\foreach \x in {0,...,3}
  \draw [thin] (1+1.5*\x,0)--(1+1.5*\x,3.25) ;

\foreach \x in {0,...,2}
  \draw [thin] (0,1+1.*\x)--(6.5,1+1.*\x) ;

\draw (3.25,2.25) node[above] {$\Delta x_i$};
\draw [>=stealth,<->] (2.5,2.25) -- (4.0,2.25) ;
\draw (2.5,0)   node[below] {$x_{i-\frac{1}{2}}$};
\draw (4.0,0)   node[below] {$x_{i+\frac{1}{2}}$};

\draw (4.5,1.5) node[right] {$\Delta y_j$};
\draw [>=stealth,<->] (4.5,1) -- (4.5,2) ;
\draw (0,1)   node[left] {$y_{j-\frac{1}{2}}$};
\draw (0,2)   node[left] {$y_{j+\frac{1}{2}}$};




\draw (3.25 -1. * 0.75, 1.5 -1. *0.5) node {$\gray\bullet$};
\draw (3.25 -1. * 0.75, 1.5 -0.447*0.5) node {$\gray\bullet$};
\draw (3.25 -1. * 0.75, 1.5 +0.447*0.5) node {$\gray\bullet$};
\draw (3.25 -1. * 0.75, 1.5 +1. *0.5) node {$\gray\bullet$};

\draw (3.25 -0.447* 0.75, 1.5 -1. *0.5) node {$\gray\bullet$};
\draw (3.25 -0.447* 0.75, 1.5 -0.447*0.5) node {$\gray\bullet$};
\draw (3.25 -0.447* 0.75, 1.5 +0.447*0.5) node {$\gray\bullet$};
\draw (3.25 -0.447* 0.75, 1.5 +1. *0.5) node {$\gray\bullet$};

\draw (3.25 +0.447* 0.75, 1.5 -1. *0.5) node {$\gray\bullet$};
\draw (3.25 +0.447* 0.75, 1.5 -0.447*0.5) node {$\gray\bullet$};
\draw (3.25 +0.447* 0.75, 1.5 +0.447*0.5) node {$\gray\bullet$};
\draw (3.25 +0.447* 0.75, 1.5 +1. *0.5) node {$\gray\bullet$};

\draw (3.25 +1.* 0.75, 1.5 -1. *0.5) node {$\gray\bullet$};
\draw (3.25 +1.* 0.75, 1.5 -0.447*0.5) node {$\gray\bullet$};
\draw (3.25 +1.* 0.75, 1.5 +0.447*0.5) node {$\gray\bullet$};
\draw (3.25 +1.* 0.75, 1.5 +1. *0.5) node {$\gray\bullet$};

\draw [densely dotted, color=gray] (3.25-0.447*0.75,1)--(3.25-0.447*0.75,2);
\draw [densely dotted, color=gray] (2.5,1.5+0.447*0.5)--(4,1.5+0.447*0.5);
\draw (3.25-0.447* 0.75,1) node[below] {$\gray k$};
\draw (2.5,1.5+0.447*0.5) node[left] {$\gray l$};


\draw (3.25,1.5) node[] {$\kappa_{ij}$};

\end{tikzpicture}
\caption{Mesh for $d=2$ with positions of quadrature points (gray bullets) in $\kappa_{ij}$ for $p=3$.}
\label{fig:stencil_2D_DGSEM}
\end{center}
\end{figure}
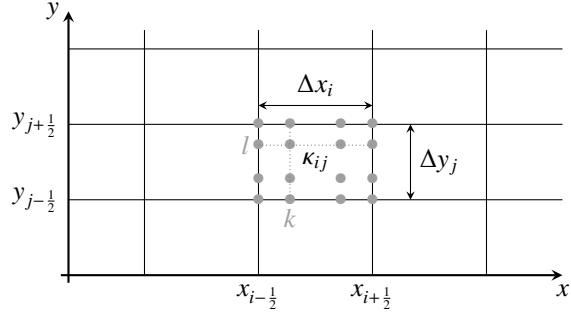

\subsection{Derivatives of the Lagrange polynomials}

It is convenient to introduce the discrete derivative matrix ${\bf D}$ \cite{kopriva_book} with entries

\begin{equation}\label{eq:nodalGL_diff_matrix}
 D_{kl} = \ell_l'(\xi_k), \quad 0\leq k,l \leq p.
\end{equation}

Note that we have $\ker{\bf D}=\mathbb{P}^0(I)$ and by the rank-nullity theorem ${\bf D}$ is of rank $p$. We will also consider ${\bf D}^{(\alpha)}$ the generalization to $\alpha$th-order derivatives:

\begin{equation}\label{eq:nodalGL_generalized_matD}
 D_{kl}^{(\alpha)} = \ell_l^{(\alpha)}(\xi_k), \quad 0\leq k,l \leq p, \quad \alpha\geq 0,
\end{equation}

\noindent with the conventions $D_{kl}^{(0)}=\ell_l(\xi_k)=\delta_{kl}$ and $D_{kl}^{(1)}=\ell_l'(\xi_k)=D_{kl}$. The matrix ${\bf D}$ maps any element of ${\cal V}_h^p$ to its derivative in ${\cal V}_h^{p-1}\subset{\cal V}_h^p$ and a direct calculation gives ${\bf D}^{(\alpha)} = {\bf D}^{\alpha}$,  and since the $(\ell_k)_{0\leq k\leq p}$ are polynomials of degree $p$, the matrix ${\bf D}$ is nilpotent:

\begin{equation}\label{eq:nilpotent_D_matrix}
 {\bf D}^{(p+1)}={\bf D}^{p+1}=0,
\end{equation}

\noindent so one can easily invert the following matrices

\begin{equation}\label{eq:inversion_ID_matrix}
 ({\bf I}-y{\bf D})^{-1} = \sum_{k=0}^p y^k{\bf D}^{(k)} \quad \forall y\in\mathbb{R},
\end{equation}

\noindent which corresponds to the truncated matrix series associated to the Taylor development of the function $x\mapsto(1-yx)^{-1}=\sum_{k\geq0}(yx)^k$ for $|xy|<1$.

Likewise ${\bf D}^{(p)}$ has columns with constant coefficients since $\ell_l^{(p)}$ is a constant function and its entries are easily obtained from \cref{eq:def_Lag_polynom}:

\begin{equation}\label{eq:entries_matDp}
 {\bf D}^{(p)}_{kl}=\ell_l^{(p)}(\xi_k)=p!\prod_{m=0,m\neq l}^p \frac{1}{\xi_l-\xi_m} \quad \forall 0\leq k,l\leq p.
\end{equation}

Integrating $\ell_k^{(\alpha)}$ over $I$ leads to the generalized integration relation

\begin{equation}\label{eq:telescoping_der}
 \sum_{l=0}^p\omega_lD_{lk}^{(\alpha)} = D_{pk}^{(\alpha-1)}-D_{0k}^{(\alpha-1)} \quad  \forall 0\leq k\leq p, \quad \alpha\geq1,
\end{equation}

\noindent which is the discrete counterpart to $\int_{-1}^1\ell^{(\alpha)}(\xi)d\xi=\ell^{(\alpha-1)}(1)-\ell^{(\alpha-1)}(-1)$ and for $\alpha=1$ we get

\begin{equation}\label{eq:INT}
    \sum \limits_{l=0}^{p} \omega_l D_{lk} = \delta_{kp}-\delta_{k0}, \quad 0 \leq k \leq p.
\end{equation}

Finally, as noticed in \cite{gassner_13}, the DGSEM satisfies the following important relation known as the summation-by-parts (SBP) property \cite{STRAND_SBP_94}

\begin{equation}\label{eq:SBP}
 \omega_kD_{kl}+\omega_lD_{lk}=\delta_{kp}\delta_{lp}-\delta_{k0}\delta_{l0}, \quad 0\leq k,l \leq p,
\end{equation}

\noindent which is the discrete counterpart to integration by parts.


%
%
\section{Time implicit discretization in one space dimension} \label{sec:DGSEM_1D_fully_discr}

We here consider \cref{eq:hyp_cons_laws} in one space dimension, $d=1$ and $f(u)=c_xu$ with $c_x>0$, over a unit domain $\Omega=(0,1)$ and consider periodic conditions $u(0,t)=u(1,t)$ which makes the analysis more difficult due to the existence of an upper block in the matrix. The present analysis however encompasses the case of Dirichlet boundary conditions leading to a block lower triangular system (see \cref{rk:Dirichlet_BC}).

\subsection{Space-time discretization}\label{sec:1d_BE_DGSEM_scheme}

The discretization in space of problem \cref{eq:hyp_cons_laws} is obtained by multiplying \cref{eq:hyp_cons_laws-a} by a test function $v_h$ in ${\cal V}_h^p$ where $u$ is replaced by the approximate solution \cref{eq:DGSEM_num_sol}, then integrating by parts in space over elements $\kappa_i$ and replacing the physical fluxes at interfaces by two-point numerical fluxes:

\begin{subequations}\label{eq:semi-discr_DGSEM}
\begin{equation}\label{eq:semi-discr_DGSEM-scheme}
 \frac{\omega_k \Delta x_i}{2}\partial_tU_i^{k} + R_i^k(u_h) = 0, \quad 1\leq i\leq N_x, \;  0\leq k\leq p, \; n\geq 0,
\end{equation}

\noindent with

\begin{equation}\label{eq:semi-discr_DGSEM-res}
 R_i^k(u_h) = -\sum_{l=0}^p\omega_lD_{lk}f(U_i^{l}) + \delta_{kp}h(U_i^{p},U_{i+1}^0) -\delta_{k0}h(U_{i-1}^{p},U_i^0),
\end{equation}
\end{subequations}

\noindent where we have used the conventions $U_{0}^{p}=U_{N_x}^{p}$ and $U_{N_x+1}^0=U_{1}^0$ to impose the periodic boundary condition. Since $c_x>0$, we use the upwind flux $h(u^-,u^+)=u^-$.

We now focus on a time implicit discretization with a backward Euler method in \cref{eq:semi-discr_DGSEM-scheme} and the fully discrete scheme reads

\begin{equation}\label{eq:discr_DGSEM_lin}
 \frac{\omega_k}{2}U_i^{k,n+1} + \lambda_i\Big( -\sum_{l=0}^p\omega_lD_{lk}U_i^{l,n+1}  + \delta_{kp}U_i^{p,n+1} - \delta_{k0}U_{i-1}^{p,n+1} \Big) = \frac{\omega_k}{2}U_i^{k,n},\quad 1\leq i\leq N_x, \;  0\leq k\leq p, \; n\geq 0,
\end{equation}

\noindent with $\lambda_i=c_x\tfrac{\Delta t^{(n)}}{\Delta x_i}$, $\Delta t^{(n)}=t^{(n+1)}-t^{(n)}>0$, with $t^{(0)}=0$, the time step, and using the notations $u_h^{(n)}(\cdot)=u_h(\cdot,t^{(n)})$ and $U_i^{k,n}=U_i^{k}(t^{(n)})$.

Summing \cref{eq:discr_DGSEM_lin} over $0\leq k\leq p$ gives

\begin{equation}\label{eq:DGSEM_cell_averaged}
 \langle u_h^{(n+1)}\rangle_i + \lambda_i\big(U_{i}^{p,n+1} - U_{i-1}^{p,n+1}\big)= \langle u_h^{(n)}\rangle_i \quad \forall 1\leq i\leq N_x, \; n\geq 0,
\end{equation}

\noindent for the cell-averaged solution

\begin{equation}\label{eq:cell-average}
 \langle u_h^{(n)}\rangle_i\coloneqq\sum_{k=0}^p\frac{\omega_k}{2}U_i^{k,n}.
\end{equation}

It is convenient to also consider \cref{eq:discr_DGSEM_lin} in vector form as

\begin{equation}\label{eq:discr_DGSEM_lin_vector_form}
 {\bf M}{\bf U}_i^{n+1} = {\bf M}{\bf U}_i^{n} + \lambda_i\Big( (2{\bf D}^\top{\bf M}-{\bf e}_p{\bf e}_p^\top){\bf U}_i^{n+1} + {\bf e}_0{\bf e}_p^\top{\bf U}_{i-1}^{n+1} \Big), \quad 1\leq i\leq N_x, \; n\geq 0.
\end{equation}

\noindent where ${\bf M}=\tfrac{1}{2}\diag(\omega_0,\dots,\omega_p)$ denotes the mass matrix, while $({\bf e}_k)_{0\leq k\leq p}$ is the canonical basis of $\mathbb{R}^{p+1}$ and ${\bf U}_i^n =(U_i^{k,n})_{0\leq k\leq p}$.

Finally, we derive the discrete counterpart to the inequality \cref{eq:PDE_entropy_ineq} for the square entropy. Left multiplying \cref{eq:discr_DGSEM_lin_vector_form} by $\big(\eta'(U_i^{0\leq k\leq p,n+1})\big)^\top={\bf U}_i^{(n+1)}$, solutions to \cref{eq:discr_DGSEM_lin} satisfy the following  inequality for the discrete square entropy $\tfrac{1}{2}\langle u_h^2\rangle_i$

\begin{equation*}
 \frac{1}{2}\langle (u_h^{(n+1)})^2\rangle_i - \frac{1}{2}\langle (u_h^{(n)})^2\rangle_i + \frac{\lambda_i}{2}\big((U_{i}^{p,n+1})^2-(U_{i-1}^{p,n+1})^2\big) \leq 0,
\end{equation*}

\noindent which brings existence and uniqueness of solutions to \cref{eq:discr_DGSEM_lin} in $L^2(\Omega_h\times\cup_{n\geq0}(t^{(n)},t^{(n+1)}),\mathbb{R})$.

\subsection{The M-matrix framework}

Before starting the analysis, we introduce the M-matrix framework that will be useful in the following. We first define the set $\mathcal{Z}^{n\times n}$ of all the $n\times n$ real matrices with nonpositive off-diagonal entries:

\begin{equation*}
    \mathcal{Z}^{n\times n} = \left\{{\bf A} =(a_{ij}) \in \mathbb{R}^{n\times n} : a_{ij} \leq 0, i \ne j \right\}.
\end{equation*}

Different characterizations of M-matrices exist \cite{PLEMMONS_Mmatrix_1977} and we use the following definition:

\begin{definition}
A matrix ${\bf A}\in \mathcal{Z}^{n\times n}$ is called an M-matrix if ${\bf A}$ is inverse-positive. That is ${\bf A}^{-1}$ exists and each entry of ${\bf A}^{-1}$ is nonnegative.
\end{definition}

We will use the following characterizations of an M-matrix \cite{PLEMMONS_Mmatrix_1977}:

\begin{theorem}\label{th:M-matrix2}
A matrix ${\bf A}\in \mathcal{Z}^{n\times n}$ is an M-matrix if and only if ${\bf A}$ is semi-positive. That is, there exists ${\bf x}=(x_1,\dots,x_n)^\top$ with $x_i>0$ such that $({\bf A}{\bf x})_i > 0$ for all $1\leq i\leq n$.
\end{theorem}


\begin{theorem}
\label{th:M-matrixCorr}
A matrix ${\bf A}\in \mathcal{Z}^{n\times n}$ is an M-matrix if ${\bf A}$ has all positive diagonal elements and it is strictly diagonally dominant, $a_{ii}>\sum_{j\neq i}|a_{ij}|$ for all $1\leq i\leq n$.
\end{theorem}

M-matrices will be used as a tool to prove positivity preservation for the DGSEM scheme which is equivalent to prove a discrete maximum principle (see \cref{th:PP_implies_MP}).

\subsection{Maximum principle for the cell average}\label{sec:MP_DGSEM_1D}

Following \cite{QinShu_impt_positive_DG_18}, we here prove in \cref{th:MP_preservation_CFL} a weaken discrete maximum principle for the cell average, $m \leq \langle u_h^{(n+1)}\rangle_i \leq M$. We then use the linear scaling limiter from \cite{zhang_shu_10a} to enforce all the DOFs at time $t^{(n+1)}$ to be in the range $[m,M]$ (see \cref{sec:pos_limiter}).

We here use the following result that shows that for the linear and conservative scheme \cref{eq:discr_DGSEM_lin}, maximum-principle preservation and positivity preservation are equivalent.

\begin{lemma}\label{th:PP_implies_MP}
 To prove a discrete maximum principle for the DGSEM scheme \cref{eq:discr_DGSEM_lin}, it is enough to prove that it is positivity preserving.
\end{lemma}

\begin{proof}
From \cref{eq:INT} we obtain 

\begin{equation*}
 \frac{\omega_k}{2} = \frac{\omega_k}{2} + \lambda_i\Big( \sum_{l=0}^p\omega_lD_{lk}  - \delta_{kp} + \delta_{k0} \Big),
\end{equation*}

\noindent and subtracting the above equation multiplied by $m$ defined in \cref{eq:PDE_max_principle} from \cref{eq:discr_DGSEM_lin}, then subtracting \cref{eq:discr_DGSEM_lin} from the above equation multiplied by $M$ in \cref{eq:PDE_max_principle}, we deduce that both $(U_{i\in\mathbb{Z}}^{0\leq k\leq p,n\geq0}-m)$ and $(M-U_{i\in\mathbb{Z}}^{0\leq k\leq p,n\geq0})$ satisfy \cref{eq:discr_DGSEM_lin}. As a consequence, the positivity preserving property, $U_{i\in\mathbb{Z}}^{0\leq k\leq p,n}\geq0$ implies $U_{i\in\mathbb{Z}}^{0\leq k\leq p,n+1}\geq0$, is equivalent to the discrete maximum principle, $m\leq U_{i\in\mathbb{Z}}^{0\leq k\leq p,n}\leq M$ implies $m\leq U_{i\in\mathbb{Z}}^{0\leq k\leq p,n+1}\leq M$. \qed
\end{proof}

Using \cref{eq:inversion_ID_matrix} with $y=2\lambda_i$ to invert \cref{eq:discr_DGSEM_lin_vector_form}, we get

\begin{equation*}
 \frac{\omega_k}{2}U_i^{k,n+1} = \sum_{l=0}^p\frac{\omega_l}{2}{\cal D}^i_{kl}U_i^{l,n} - \lambda_i\big( {\cal D}^i_{kp}U_i^{p,n+1} - {\cal D}^i_{k0}U_{i-1}^{p,n+1} \big),
\end{equation*}

\noindent where the ${\cal D}^i_{kl}$ denote the entries of the matrix

\begin{equation}\label{eq:def_calD}
 {\mathbfcal D}_i \coloneqq \big({\bf I} - 2\lambda_i{\bf D}^\top\big)^{-1} \overset{\cref{eq:inversion_ID_matrix}}{=} \sum_{l=0}^p (2\lambda_i{\bf D}^\top)^l.
\end{equation}

We use \cref{eq:DGSEM_cell_averaged} to get $\lambda_i U_{i-1}^{p,n+1}=\lambda_i U_{i}^{p,n+1}+\langle u_h^{(n+1)}\rangle_i - \langle u_h^{(n)}\rangle_i$ and injecting this result into the above expression for $U_i^{p,n+1}$ gives

\begin{equation*}
 \sigma_i^pU_i^{p,n+1} = \xi_i^{p,n} + 2 {\cal D}^i_{p0} \Big( \langle u_h^{(n+1)}\rangle_i - \langle u_h^{(n)}\rangle_i \Big),
\end{equation*}

\noindent where

\begin{equation}\label{eq:coeffs_U_ip_n+1}
 \sigma_i^p = \omega_p + 2\lambda_i({\cal D}^i_{pp}-{\cal D}^i_{p0}), \quad \xi_i^{p,n}  = \sum_{l=0}^p\omega_l{\cal D}^i_{pl}U_i^{l,n}.
\end{equation}

Further using the above expression to eliminate $U_i^{p,n+1}$ and $U_{i-1}^{p,n+1}$ from the cell-averaged scheme \cref{eq:DGSEM_cell_averaged}, we finally obtain

\begin{equation} \label{eq:lin_cell_average_relation}
\begin{aligned}
 \left(1+2\lambda_i\frac{{\cal D}^i_{p0}}{\sigma_i^p}\right)\langle u_h^{(n+1)}\rangle_{i} - 2\lambda_i\frac{{\cal D}^{i-1}_{p0}}{\sigma_{i-1}^p}\langle u_h^{(n+1)}\rangle_{i-1} =& \left(1+2\lambda_i\frac{{\cal D}^i_{p0}}{\sigma_i^p}\right)\langle u_h^{(n)}\rangle_{i} - 2\lambda_i\frac{{\cal D}^{i-1}_{p0}}{\sigma_{i-1}^p}\langle u_h^{(n)}\rangle_{i-1} - \lambda_i\left(\frac{\xi_i^{p,n}}{\sigma_i^p} - \frac{\xi_{i-1}^{p,n}}{\sigma_{i-1}^p}\right) \\
 =& \langle u_h^{(n)}\rangle_{i}
 - \lambda_i\frac{\xi_i^{p,n}-2{\cal D}^i_{p0}\langle u_h^{(n)}\rangle_{i}}{\sigma_i^p} + \lambda_i\frac{\xi_{i-1}^{p,n}-2{\cal D}^{i-1}_{p0}\langle u_h^{(n)}\rangle_{i-1}}{\sigma_{i-1}^p} \\
 =& \sum_{k=0}^p\frac{\omega_k}{2}\left(\left(1-\frac{2\lambda_i\left({\cal D}^i_{pk}-{\cal D}^i_{p0}\right)}{\sigma_i^p}\right)U_{i}^{k,n} + \left(\frac{2\lambda_i\left({\cal D}^{i-1}_{pk}-{\cal D}^{i-1}_{p0}\right)}{\sigma_{i-1}^p}\right)U_{i-1}^{k,n}\right),
\end{aligned}
\end{equation}

\noindent where we have used \cref{eq:cell-average,eq:coeffs_U_ip_n+1} in the last step.

Let us now derive conditions on $\lambda_i$ for which the above relation preserves a discrete maximum principle for the cell-averaged solution. According to \cref{th:PP_implies_MP}, it is enough to prove that the scheme preserves positivity, i.e., $U_{1\leq i\leq N_x}^{0\leq k\leq p,n}\geq0$ imply $\langle u_h^{(n+1)}\rangle_{1\leq i\leq N_x}\geq0$. We will thus show that, under some conditions on $\lambda_i$, the matrix stemming from the linear system \cref{eq:lin_cell_average_relation} for the $\langle u_h^{(n+1)}\rangle_{1\leq j\leq N_x}$ is an M-matrix and that its RHS is a nonnegative combination of the $U_{i}^{k,n}$ and $U_{i-1}^{k,n}$. Assuming the DOFs at time $t^{(n)}$ are in the range $[m,M]$, so will do the cell-averaged solutions $\langle u_h^{(n+1)}\rangle_{1\leq i\leq N_x}$.

In view of \cref{eq:lin_cell_average_relation}, conditions for the RHS to be nonnegative read

\begin{equation}\label{eq:MPP_cond_i}
 \sigma_i^p - 2\lambda_i({\cal D}^i_{pk}-{\cal D}^i_{p0})=\omega_p+2\lambda_i({\cal D}^i_{pp}-{\cal D}^i_{pk}) \geq 0, \quad {\cal D}^i_{pk}-{\cal D}^i_{p0} \geq 0 \quad \forall 0\leq k\leq p, \; 1\leq i\leq N_x,
\end{equation}

\noindent with $\sigma_i^p>0$, while we impose the off-diagonal entries to be negative through

\begin{equation}\label{eq:MPP_cond_ii}
 \sigma_i^p = \omega_p + 2\lambda_i({\cal D}^i_{pp}-{\cal D}^i_{p0}) > 0, \quad {\cal D}^i_{p0} \geq 0.
\end{equation}

The strict inequality on the ${\cal D}^i_{p0}$ allows to satisfy \cref{th:M-matrix2} by choosing the vector $\mathbf{x}$ such that

\begin{equation*}
    x_i = \prod_{j=1, j\ne i}^{N_x} \frac{{\cal D}^j_{p0}}{\sigma_j^p} > 0 \quad \forall 1\leq i \leq N_x,
\end{equation*}

\noindent and we obtain $(\mathbf{Ax})_i = x_i > 0$ from \cref{eq:MPP_cond_ii}.


%
\begin{lemma}\label{th:lemme_lambda_min}
 For all $p\geq1$, there exists a finite $\lambda_{min}=\lambda_{min}(p)\geq0$ such that conditions \cref{eq:MPP_cond_i,eq:MPP_cond_ii} are satisfied for all $\lambda_i>\lambda_{min}$, $1\leq i\leq N_x$.
\end{lemma}

\begin{proof}
 Let consider the first condition in \cref{eq:MPP_cond_i}, similar arguments hold for all other conditions. For a fixed $0\leq k\leq p$, use \cref{eq:def_calD} to rewrite

\begin{equation*}
 {\cal D}^i_{pp}-{\cal D}^i_{pk} = \sum_{l=0}^p(2\lambda_i)^l\left(D_{pp}^{(l)}-D_{kp}^{(l)}\right) = \sum_{l=0}^{p-1}(2\lambda_i)^l\left(D_{pp}^{(l)}-D_{kp}^{(l)}\right),
\end{equation*}

\noindent since by \cref{eq:entries_matDp}, we have $D^{(p)}_{kp}=D^{(p)}_{pp}$. Hence for large $\lambda_i$, we have

\begin{equation*}
{\cal D}^i_{pp}-{\cal D}^i_{pk} \underset{\lambda_i}{\sim} (2\lambda_i)^{p-1}\left(D_{pp}^{(p-1)}-D_{kp}^{(p-1)}\right)
\end{equation*}

\noindent and we are going to show that this is a positive quantity. By using the linearity of $\ell_p^{(p-1)}(\cdot)$, we have $D_{kp}^{(p-1)}=\tfrac{1-\xi_k}{2}D_{0p}^{(p-1)}+\tfrac{1+\xi_k}{2}D_{pp}^{(p-1)}$. Then, we obtain 

\begin{align*}
 D_{pp}^{(p-1)}-D_{kp}^{(p-1)} = \frac{1-\xi_k}{2}\left(D_{pp}^{(p-1)}-D_{0p}^{(p-1)}\right) 
 \overset{\cref{eq:telescoping_der}}{=} \frac{1-\xi_k}{2}\sum_{l=0}^p\omega_lD_{lp}^{(p)} \overset{\cref{eq:entries_matDp}}{=} (1-\xi_k)D_{pp}^{(p)} \overset{\cref{eq:entries_matDp}}{=} \frac{(1-\xi_k)p!}{\prod\limits_{l=0}^{p-1}(1-\xi_l)}>0,
\end{align*}

\noindent which concludes the proof. \qed
\end{proof}

The following theorem immediately follows.

\begin{theorem}\label{th:MP_preservation_CFL}
 Under the conditions $\lambda_{1\leq i\leq N_x}>\lambda_{min}$ defined in \cref{th:lemme_lambda_min}, the DGSEM scheme \cref{eq:discr_DGSEM_lin} is maximum principle preserving for the cell-averaged solution:

 \begin{equation*}
  m\leq U_{i}^{k,n} \leq M \quad \forall 1\leq i\leq N_x, \; 0\leq k\leq p \quad\Rightarrow\quad m \leq \avg{u_h^{(n+1}}_i \leq M \quad \forall 1\leq i\leq N_x.
 \end{equation*}
\end{theorem}

\Cref{tab:lower_bounds_CFL} indicates the lower bounds on the $\lambda_i$ as a function of the polynomial degree $p$ evaluated from the conditions \cref{eq:MPP_cond_i,eq:MPP_cond_ii}. We observe that the second-order in space scheme, $p=1$, is unconditionally maximum principle preserving, while the lower bound decreases with increasing $p$ values for $p\geq2$. These bounds are different from those obtained in \cite[Tab.~1]{QinShu_impt_positive_DG_18} for the modal DG scheme with Legendre polynomials as function basis. In particular, the modal DG scheme with $p=1$ is not unconditionally maximum principle preserving and is seen to require a larger CFL value for larger $p$ values.

\begin{table}
  \begin{center}
  \caption{Lower bounds on the non-dimensional time step $\lambda_i>\lambda_{min}$, $1\leq i\leq N_x$, for \cref{eq:MPP_cond_i,eq:MPP_cond_ii} to hold, which make \cref{eq:lin_cell_average_relation} maximum principle preserving.}
  \begin{tabular}{lllllll}
    \hline
		p & 1 & 2 & 3 & 4 & 5 & 6\\
    \hline
		$\lambda_{min}$ & $0$ & $\tfrac{1}{4}$ & $\tfrac{1+\sqrt{5}}{6(5-\sqrt{5})}$ & $0.150346$ & $0.147568$ & $0.109977$\\
    \hline
  \end{tabular}
  \label{tab:lower_bounds_CFL}
  \end{center}
\end{table}

\begin{remark}[Linear hyperbolic systems]
The above results also apply to the case of linear hyperbolic systems of size $n_{eq}$ with constant coefficients $\partial_t{\bf u}+{\bf A}\partial_x{\bf u}=0$ with ${\bf A}$ diagonalizable in $\mathbb{R}$ with eigenvalues $\psi_k$, normalized left and right eigenvectors ${\bf l}_k$ and ${\bf r}_k$ such that ${\bf l}_k^\top{\bf r}_l=\delta_{kl}$. Assuming that the right eigenvectors form a basis of $\mathbb{R}^{n_{eq}}$ and setting ${\bf u}=\sum_ku_k{\bf r}_k$, each component satisfies a maximum principle:  $\partial_tu_k+\psi_k\partial_xu_k=0$. Using a Roe flux, ${\bf h}({\bf u}^-,{\bf u}^+)=\tfrac{1}{2}{\bf A}({\bf u}^-+{\bf u}^+)+\tfrac{1}{2}\sum_k|\psi_k|(u_k^--u_k^+){\bf r}_k$, the time implicit DGSEM \cref{eq:discr_DGSEM_lin} decouples into $n_{eq}$ independent schemes \cref{eq:discr_DGSEM_lin} for $u_k$ upon left multiplication by ${\bf l}_k$ since ${\bf l}_k^\top{\bf h}({\bf u}^-,{\bf u}^+)=\tfrac{\psi_k}{2}(u_k^-+u_k^+)+\tfrac{|\psi_k|}{2}(u_k^--u_k^+)$ reduces to the upwind flux. \qed
\end{remark}

\begin{remark}[Geometric source]\label{rk:adv_reac_pbs}
\Cref{th:MP_preservation_CFL} with the bounds from \cref{th:lemme_lambda_min} also applies to linear equations with a geometric source term:

\begin{equation*}
 \partial_tu + c_x\partial_xu = s(x) \quad \text{in } \Omega\times(0,\infty), 
\end{equation*}

\noindent with $s(\cdot)\geq0$. Providing nonnegative initial and boundary data are imposed, the entropy solution remains nonnegative for all time. The DGSEM scheme \cref{eq:discr_DGSEM_lin} for the discretization of the above equation remains the same by substituting $U_i^{k,n}+s(x_i^k)\Delta t$ for $U_i^{k,n}$ in the RHS. The conditions to obtain an M-matrix in  \cref{eq:lin_cell_average_relation} are therefore unchanged and only the RHS is modified by the above change of variable on $U_i^{k,n}$. We thus conclude that the present DGSEM will be positivity-preserving for the cell-averaged solution under the same conditions as in \cref{th:MP_preservation_CFL}. \qed
\end{remark}

\subsection{Linear system solution}\label{sec:1D_linear_sys_solution}

Equations \cref{eq:discr_DGSEM_lin} or \cref{eq:discr_DGSEM_lin_vector_form} result in a banded block linear system

\begin{equation}\label{eq:global_lin_sys_1D}
 \mathbb{A}_{1d}{\bf U}^{(n+1)} = \mathbb{M}_{1d}{\bf U}^{(n)}
\end{equation}

\noindent of size $N_x(p+1)$ with blocks of size $p+1$ to be solved for ${\bf U}^{(n+1)}$ where ${\bf U}_{1+k+(p+1)j}=U_i^k$ and $\mathbb{M}_{1d}$ is the global mass matrix. Using the block structure of \cref{eq:global_lin_sys_1D} is usually important for its efficient resolution with either direct or iterative block-based solvers. We will also propose a direct algorithm below based on the inversion of the diagonal block only. In most cases, we need to invert the diagonal blocks and we now prove that they are unconditionally invertible and give an explicit expression for their inverse. Unless necessary, we omit the cell index $1\leq i\leq N_x$ in this section. 

\begin{lemma}
 For all $p\geq1$, the diagonal blocks ${\bf L}_{1d}{\bf M}$ of $\mathbb{A}_{1d}$ in the linear system \cref{eq:discr_DGSEM_lin_vector_form}, with

\begin{equation}\label{eq:1D_diag_blocks}
 {\bf L}_{1d} = {\bf I}-2\lambda\mathbfcal{L}, \quad \mathbfcal{L}={\bf D}^\top-\frac{1}{\omega_p}{\bf e}_p{\bf e}_p^\top, 
\end{equation}

\noindent are invertible for any $\lambda>0$.
\end{lemma}

\begin{proof}
Let us prove that ${\bf L}_{1d}$ is invertible: assume that ${\bf L}_{1d}{\bf u}=0$ for some ${\bf u}=(u_0,\dots,u_p)^\top$, then by \cref{eq:1D_diag_blocks} we have

\begin{equation*}
 ({\bf I}-2\lambda{\bf D}^\top){\bf u} = -\frac{2\lambda}{\omega_p}{\bf e}_p{\bf e}_p^\top{\bf u} = -\frac{2\lambda u_p}{\omega_p}{\bf e}_p \quad \Rightarrow \quad {\bf u}=-\frac{2\lambda u_p}{\omega_p}{\mathbfcal D}{\bf e}_p,
\end{equation*}

\noindent with ${\mathbfcal D}=({\bf I}-2\lambda{\bf D}^\top)^{-1}$ given by \cref{eq:def_calD}. Hence $u_k=-\tfrac{2\lambda u_p}{\omega_p}{\cal D}_{kp}$ and for $k=p$ we get $(\omega_p+2\lambda{\cal D}_{pp})u_p=0$, so $u_p=0$ since $\omega_p+2\lambda{\cal D}_{pp}>0$ from \cref{eq:MPP_cond_ii} and we conclude that ${\bf u}=0$. Note that \cref{eq:1D_diag_blocks} is invertible for all $\lambda>0$ since we have ${\cal D}_{pp}=1+\sum_{l=1}^p(2\lambda)^lD_{pp}^{(l)}>0$. Indeed, by differentiating \cref{eq:def_Lag_polynom} $l$-times we obtain

\begin{equation*}
 D_{pp}^{(l)} = \ell^{(l)}_p(1) = \sum_{k_1=0}^p\sum_{k_2=0,k_2\neq k_1}^p\cdots \sum_{k_l=1,k_l\notin\{k_1,\dots,k_{l-1}\}}^p \prod_{m=1}^l\frac{1}{1-\xi_{k_m}}>0 \quad \forall 1\leq l\leq p.
\end{equation*}

\qed
\end{proof}

Note that we easily deduce the explicit expression of the inverse of \cref{eq:1D_diag_blocks} from \cref{eq:def_calD} by using the Sherman-Morisson formula which provides the inverse of the sum of an invertible matrix ${\bf A}$ and a rank-one matrix ${\bf u}{\bf v}^\top$:

\begin{equation*}
 \big({\bf A}+{\bf u}{\bf v}^\top\big)^{-1}
 = \big({\bf I}+{\bf A}^{-1}{\bf u}{\bf v}^\top\big)^{-1}{\bf A}^{-1}
 = \big({\bf I}-{\bf A}^{-1}{\bf u}{\bf v}^\top+{\bf A}^{-1}{\bf u}{\bf v}^\top{\bf A}^{-1}{\bf u}{\bf v}^\top-\dots\big){\bf A}^{-1}
 = \left({\bf I}-\frac{1}{{1 + \bf v}^\top{\bf A}^{-1}{\bf u}}{\bf A}^{-1}{\bf u}{\bf v}^\top\right){\bf A}^{-1}.
\end{equation*}

Using ${\bf A}={\bf I}-2\lambda{\bf D}^\top$ and ${\bf u}={\bf v}={\bf e}_p$, we obtain

\begin{equation*}
 {\bf M}^{-1}{\bf L}_{1d}^{-1} = {\bf M}^{-1}\left({\bf I}-2\lambda\left({\bf D}^\top-\frac{1}{\omega_p}{\bf e}_p{\bf e}_p^\top\right)\right)^{-1} = {\bf M}^{-1}\left({\bf I}-\frac{2\lambda}{\omega_p+2\lambda{\cal D}^i_{pp}}{\mathbfcal D}{\bf e}_p{\bf e}_p^\top\right){\mathbfcal D}
\end{equation*}

\noindent with ${\mathbfcal D}=({\bf I}-2\lambda{\bf D}^\top)^{-1}$ given by \cref{eq:def_calD}. Again, this formula is well defined since $\omega_p+2\lambda{\cal D}^i_{pp}>0$ by \cref{eq:MPP_cond_ii}.

Let us finally propose a method to solve the global linear system \cref{eq:global_lin_sys_1D}. From \cref{eq:discr_DGSEM_lin}, we observe that $\mathbb{A}_{1d}=\mathbb{A}_0-\lambda_1{\bf e}_1^0({\bf e}_{N_x}^p)^\top$ in \cref{eq:global_lin_sys_1D} with $\mathbb{A}_0$ a block lower triangular matrix and a rank-one matrix defined from $({\bf e}_i^k)_{1\leq i\leq N_x}^{0\leq k\leq p}$ the canonical basis of $\mathbb{R}^{N_x(p+1)}$. Using again the Sherman-Morisson formula, we easily solve \cref{eq:global_lin_sys_1D} from \cref{algo:invert_A1d} where steps 1 and 2 can be solved efficiently using blockwise back substitution.

%
\begin{algorithm}
\caption{Algorithm flowchart for solving the global system \cref{eq:global_lin_sys_1D} by using the decomposition $\mathbb{A}_{1d}=\mathbb{A}_0-\lambda_1{\bf e}_1^0({\bf e}_{N_x}^p)^\top$ with $\mathbb{A}_0$ a block lower triangular matrix and ${\bf e}_1^0({\bf e}_{N_x}^p)^\top$ a rank-one matrix.}\label{algo:invert_A1d}
\begin{algorithmic}[1]
\STATE{solve $\mathbb{A}_0{\bf V}=\mathbb{M}_{1d}{\bf U}^{(n)}$ for ${\bf V}\in\mathbb{R}^{N_x(p+1)}$};
\STATE{solve $\mathbb{A}_0{\bf W}={\bf e}_1^0$ for ${\bf W}\in\mathbb{R}^{N_x(p+1)}$};
\STATE{set ${\bf U}^{(n+1)}={\bf V} + \tfrac{\lambda_1{\bf e}_{N_x}^p\cdot{\bf V}}{1-\lambda_1{\bf e}_{N_x}^p\cdot{\bf W}}{\bf W}$}.
\end{algorithmic}
\end{algorithm}

Applying \cref{algo:invert_A1d} again requires $1-\lambda_1{\bf e}_{N_x}^p\cdot{\bf W}=1-\lambda_1{\bf e}_{N_x}^p\cdot(\mathbb{A}_0^{-1}{\bf e}_1^0)\neq0$. This is indeed the case and to prove it we temporarily consider a uniform mesh for the sake of clarity, so $\lambda_i=\lambda$. We observe that the solution to $\mathbb{A}_0{\bf W}={\bf e}_1^0$ satisfies ${\bf L}_{1d}{\bf M}{\bf W}_1={\bf e}_0$ and ${\bf L}_{1d}{\bf M}{\bf W}_i=(\lambda{\bf e}_0{\bf e}_p^\top){\bf W}_{i-1}$ for $j\geq2$. We thus get ${\bf W}_i=(\lambda{\bf M}^{-1}{\bf L}_{1d}^{-1}{\bf e}_0{\bf e}_p^\top)^{j-1}{\bf M}^{-1}{\bf L}_{1d}^{-1}{\bf e}_0$ and ${\bf e}_{N_x}^p\cdot{\bf W}=\lambda^{N_x-1}\big(\tfrac{1}{\omega_p}({\bf L}_{1d}^{-1})_{p0}\big)^{N_x}$ with $({\bf L}_{1d}^{-1})_{p0}=\tfrac{2}{\omega_p}(1-\tfrac{2\lambda}{\omega_p+2\lambda{\cal D}^i_{pp}}{\cal D}^i_{pp}){\cal D}^i_{p0}=\tfrac{2\lambda{\cal D}^i_{p0}}{\omega_p+2\lambda{\cal D}^i_{pp}}>0$ from \cref{eq:MPP_cond_ii}. Note that ${\cal D}^i_{p0}>0$ holds for $\lambda>\lambda_{min}$ defined in \cref{th:lemme_lambda_min}.

\begin{remark}[Dirichlet boundary condition]\label{rk:Dirichlet_BC}
The case of an inflow boundary condition, $u(0,t)=g(t)\in[m,M]$, results in a similar linear system \cref{eq:global_lin_sys_1D} with the only difference that $U_{-1}^{p,n+1}=g(t^{(+1)})$ in \cref{eq:semi-discr_DGSEM-res}. As a consequence, \cref{eq:global_lin_sys_1D} is a block lower triangular matrix with the same diagonal blocks ${\bf L}_{1d}$ as in \cref{eq:discr_DGSEM_lin_vector_form}. The system \cref{eq:global_lin_sys_1D} is therefore easily solved by block forward substitution since the diagonal blocks are invertible. Likewise, the cell-averaged solution is maximum principle preserving under the same conditions in \cref{th:lemme_lambda_min} as with periodic boundary conditions. \qed
\end{remark}

%
%
\section{Time implicit discretization in two space dimensions}\label{sec:DGSEM_2D_fully_discr}

We now consider a 2D linear problem with constant coefficients:

\begin{subequations}\label{eq:hyp_2Dcons_laws}
\begin{align}
 \partial_tu + c_x\partial_xu + c_y\partial_y u &= 0, \quad \mbox{in }\Omega\times(0,\infty), \label{eq:hyp_2Dcons_laws-a} \\
 u(\cdot,0) &= u_{0}(\cdot),\quad\mbox{in }\Omega, \label{eq:hyp_2Dcons_laws-b}
\end{align}
\end{subequations}

\noindent with boundary conditions on $\partial\Omega$ and we again assume $c_x\geq0$ and $c_y\geq0$ without loss of generality. We again assume $\Omega=\mathbb{R}^2$ for the analysis, which amounts in practice to consider a rectangular domain with periodic boundary conditions. As in the 1D case, considering inflow and outflow boundary conditions results in a block lower triangular system to be solved and hence an easier analysis. The results in this section may be easily generalized to three space dimensions and \cref{app:3D_DGSEM_scheme} summarizes the analysis of the three-dimensional scheme.

\subsection{Space-time discretization}

We consider a Cartesian mesh with rectangular elements of measure $|\kappa_{ij}|=\Delta x_i\times \Delta y_j$ for all $i,j$ in $\mathbb{Z}$. Using again a time implicit discretization with a backward Euler method and upwind numerical fluxes, the fully discrete scheme reads

\begin{equation}\label{eq:fully-discr_DGSEM_2d}
\begin{aligned}
  \frac{\omega_k\omega_l}{4}(U_{ij}^{kl,n+1}-U_{ij}^{kl,n}) -& \frac{\omega_l}{2}\lambda_{x_i}\Big(\sum_{m=0}^p \omega_mD_{mk}U_{ij}^{ml,n+1} - \delta_{kp}U_{ij}^{pl,n+1} + \delta_{k0}U_{(i-1)j}^{pl,n+1}\Big)\\
  -& \frac{\omega_k}{2}\lambda_{y_j}\Big(\sum_{m=0}^p \omega_mD_{ml}U_{ij}^{km,n+1} - \delta_{lp}U_{ij}^{kp,n+1} + \delta_{l0}U_{i(j-1)}^{kp,n+1}\Big) = 0,
\end{aligned}
\end{equation}

\noindent where $\lambda_{x_i}=\tfrac{c_x\Delta t}{\Delta x_i}$ and  $\lambda_{y_j}=\tfrac{c_y\Delta t}{\Delta y_j}$. We again use the conventions $U_{0j}^{pl}=U_{N_xj}^{pl}$ and $U_{i0}^{kp}=U_{iN_y}^{kp}$ to take the periodic boundary conditions into account. Using a vector storage of the DOFs as $({\bf U}_{ij})_{n_{kl}}=U_{ij}^{kl}$ with $1\leq n_{kl}\coloneqq1+k+l(p+1)\leq N_p$ and $N_p=(p+1)^2$, it will be convenient to rewrite the scheme under vector form as

\begin{equation}\label{eq:2D_discr_DGSEM_lin_vector_form}
\begin{aligned}
 ({\bf M}\otimes{\bf M})({\bf U}_{ij}^{n+1} -{\bf U}_{ij}^{n}) &- \lambda_{x_i}\Big({\bf M}\otimes(2{\bf D}^\top{\bf M}-{\bf e}_{p}{\bf e}_{p}^\top)\Big){\bf U}_{ij}^{n+1} - \lambda_{x_i}\big({\bf M}\otimes{\bf e}_{0}{\bf e}_{p}^\top\big){\bf U}_{(i-1)j}^{n+1}\\
  &- \lambda_{y_j}\Big((2{\bf D}^\top{\bf M}-{\bf e}_{p}{\bf e}_{p}^\top)\otimes{\bf M}\Big){\bf U}_{ij}^{n+1} - \lambda_{y_j}\big({\bf e}_{0}{\bf e}_{p}^\top\otimes{\bf M}\big){\bf U}_{i(j-1)}^{n+1} = 0,
\end{aligned}
\end{equation}

\noindent where ${\bf M}=\tfrac{1}{2}\diag(\omega_0,\dots,\omega_p)$ denotes the 1D mass matrix, ${\bf M}\otimes{\bf M}$ the 2D mass matrix, $({\bf e}_k)_{0\leq k\leq p}$ is the canonical basis of $\mathbb{R}^{p+1}$, and $\otimes$ denotes the Kronecker product \cite{VanLoan2016,vanLoanKron00}: $({\bf A}\otimes{\bf B})_{n_{kl}n_{k'l'}}={\bf A}_{ll'}{\bf B}_{kk'}$, which satisfies 

\begin{subequations}
\begin{equation}\label{eq:prop_Kronecker_prod}
 ({\bf A}\otimes{\bf B})({\bf C}\otimes{\bf D})={\bf A}{\bf C}\otimes{\bf B}{\bf D}, \quad ({\bf A}\otimes{\bf B})^{-1}={\bf A}^{-1}\otimes{\bf B}^{-1}, \quad ({\bf A}\otimes{\bf B})^\top={\bf A}^\top\otimes{\bf B}^\top.
\end{equation}

Likewise, for diagonalizable matrices ${\bf A}={\bf R}_A\Psib_A{\bf R}_A^{-1}$ and ${\bf B}={\bf R}_B\Psib_B{\bf R}_B^{-1}$, the product ${\bf A}\otimes{\bf B}$ is also diagonalizable with eigenvalues being the product of eigenvalues of ${\bf A}$ and ${\bf B}$:

\begin{equation}\label{eq:prop_Kronecker_vap}
 {\bf A}\otimes{\bf B} = ({\bf R}_A\otimes{\bf R}_B)(\Psib_A\otimes\Psib_B)({\bf R}_A\otimes{\bf R}_B)^{-1}.
\end{equation}
\end{subequations}

Summing \cref{eq:fully-discr_DGSEM_2d} over $0 \leq k,l \leq p$ gives:

\begin{equation}\label{eq:DGSEM_cell_averaged_2d}
 \langle u_h^{(n+1)}\rangle_{ij} - \langle u_h^{(n)}\rangle_{ij}  + \frac{\lambda_{x_i}}{2}\sum_{l=0}^{p} \omega_l \Big(U_{ij}^{pl,n+1}-U_{(i-1)j}^{pl,n+1}\Big)  + \frac{\lambda_{y_j}}{2} \sum_{k=0}^{p} \omega_k \Big(U_{ij}^{kp,n+1}-U_{i(j-1)}^{kp,n+1}\Big)= 0,
\end{equation}

\noindent where the cell-average operator reads

\begin{equation}\label{eq:DGSEM_cell_averaged_sol_2d}
 \langle u_h\rangle_{ij} = \sum_{k=0}^p\sum_{l=0}^p\frac{\omega_k\omega_l}{4}U_{ij}^{kl}.
\end{equation}

Finally, left-multiplying \cref{eq:2D_discr_DGSEM_lin_vector_form} by ${\bf U}_{ij}^{(n+1)}$ brings $L^2$ stability:

\begin{align*}
 \frac{1}{2}\langle(u_h^{(n+1)})^2\rangle_{ij} - \frac{1}{2}\langle(u_h^{(n)})^2\rangle_{ij} + \frac{\lambda_{x_i}}{2}\sum_{l=0}^p\frac{\omega_l}{2}\big((U_{ij}^{pl,n+1})^2-(U_{(i-1)j}^{pl,n+1})^2\big) + \frac{\lambda_{y_j}}{2}\sum_{k=0}^p\frac{\omega_k}{2}\big((U_{ij}^{kp,n+1})^2-(U_{i(j-1)}^{kp,n+1})^2\big) \leq 0.
\end{align*}

The discrete derivative matrix is still nilpotent as shown in \cref{app:2D_difference_matrix}. Unfortunately, the scheme \cref{eq:2D_discr_DGSEM_lin_vector_form} is in general not maximum principle preserving for the cell average as may be observed in the numerical experiments of \cref{sec:num_xp}. We now propose to modify the scheme by adding graph viscosity to make it maximum principle preserving.

\subsection{Maximum principle through graph viscosity}


We add a graph viscosity \cite{guermond_popov_GV_16} term ${\bf V}_{ij}^{(n+1)}$ to the LHS of \cref{eq:2D_discr_DGSEM_lin_vector_form} which becomes

\begin{equation}\label{eq:2D_discr_DGSEM_with_GV_lin_vector_form}
\begin{aligned}
 ({\bf M}\otimes{\bf M})({\bf U}_{ij}^{n+1} -{\bf U}_{ij}^{n}) &- \lambda_{x_i}\big({\bf M}\otimes(2{\bf D}^\top{\bf M}-{\bf e}_{p}{\bf e}_{p}^\top)\big){\bf U}_{ij}^{n+1} - \lambda_{x_i}({\bf M}\otimes{\bf e}_{0}{\bf e}_{p}^\top){\bf U}_{(i-1)j}^{n+1} \\
  &- \lambda_{y_j}\big((2{\bf D}^\top{\bf M}-{\bf e}_{p}{\bf e}_{p}^\top)\otimes{\bf M}\big){\bf U}_{ij}^{n+1} - \lambda_{y_j}({\bf e}_{0}{\bf e}_{p}^\top\otimes{\bf M}){\bf U}_{i(j-1)}^{n+1} + {\bf V}_{ij}^{(n+1)} = 0,
\end{aligned}
\end{equation}

\noindent where

\begin{equation}\label{eq:2D_graph_viscosity}
\begin{aligned}
 {\bf V}_{ij}^{(n+1)} &= 2d_{ij}\Big(\lambda_{x_i}{\bf M}\otimes\big({\bf M}-\omegab{\bf 1}^\top{\bf M}\big)+\lambda_{y_j}\big({\bf M}-\omegab{\cal\bf 1}^\top{\bf M}\big)\otimes{\bf M}\Big){\bf U}_{ij}^{(n+1)}\\
 &= 2d_{ij}\Big((\lambda_{x_i}+\lambda_{y_j}){\bf I}\otimes{\bf I}-\lambda_{x_i}({\bf I}\otimes\omegab{\bf 1}^\top)-\lambda_{y_j}(\omegab{\bf 1}^\top\otimes{\bf I})\Big)({\bf M}\otimes{\bf M}){\bf U}_{ij}^{(n+1)},
\end{aligned}
\end{equation}

\noindent with $d_{ij}\geq0$, $\omegab=\tfrac{1}{2}(\omega_0,\dots,\omega_p)^\top$ and ${\bf 1}=(1,\dots,1)^\top\in\mathbb{R}^{p+1}$, which reads componentwise as

\begin{equation}\label{eq:graph_viscosity_compnentwise}
 V_{ij}^{kl,n+1} = d_{ij}\frac{\omega_k\omega_l}{2}\Big(\lambda_{x_i}\sum_{m=0}^p\frac{\omega_m}{2}\big(U_{ij}^{kl,n+1}-U_{ij}^{ml,n+1}\big) + \lambda_{y_j}\sum_{m=0}^p\frac{\omega_m}{2}\big(U_{ij}^{kl,n+1}-U_{ij}^{km,n+1}\big)\Big).
\end{equation}

This term keeps conservation of the scheme: $\sum_{k,l}V_{ij}^{kl,n+1}=0$, so the cell-averaged scheme still satisfies \cref{eq:DGSEM_cell_averaged_2d}. It also enforces the $L^2$ stability since

\begin{align*}
 {\bf U}_{ij}\cdot{\bf V}_{ij} \overset{\cref{eq:graph_viscosity_compnentwise}}{=}& d_{ij}\sum_{k,l=0}^p\frac{\omega_k\omega_l}{2}U_{ij}^{kl}\left(\lambda_{x_i}\sum_{m=0}^p\frac{\omega_m}{2}\big(U_{ij}^{kl}-U_{ij}^{ml}\big) + \lambda_{y_j}\sum_{m=0}^p\frac{\omega_m}{2}\big(U_{ij}^{kl}-U_{ij}^{km}\big)\right) \\
=& d_{ij}\sum_{k,l=0}^p\frac{\omega_k\omega_l}{2}\left(\lambda_{x_i}\sum_{m=0}^p\frac{\omega_m}{2}\frac{(U_{ij}^{kl}-U_{ij}^{ml})^2 + (U_{ij}^{kl})^2-(U_{ij}^{ml})^2}{2} + \lambda_{y_j}\sum_{m=0}^p\frac{\omega_m}{2} \frac{(U_{ij}^{kl}-U_{ij}^{km})^2 + (U_{ij}^{kl})^2-(U_{ij}^{km})^2}{2} \right) \\
=& d_{ij}\sum_{k,l=0}^p\frac{\omega_k\omega_l}{2}\left(\lambda_{x_i}\sum_{m=0}^p\frac{\omega_m}{2}\frac{(U_{ij}^{kl}-U_{ij}^{ml})^2}{2} + \lambda_{y_j}\sum_{m=0}^p\frac{\omega_m}{2}\frac{(U_{ij}^{kl}-U_{ij}^{km})^2}{2} \right) \geq 0.
\end{align*}

We now look for conditions on the linear system \cref{eq:2D_discr_DGSEM_with_GV_lin_vector_form} to correspond to an M-matrix, thus imposing a maximum principle for the DOFs.

\begin{lemma}\label{th:Mmat_condition_on_GV}
Under the condition

\begin{equation}\label{eq:cond_on_dij_for_M_matrix}
 d_{ij}\geq 2\max_{0\leq k \neq m \leq p}\bigg(-\frac{D_{mk}}{\omega_k}\bigg),
\end{equation}

\noindent the linear system \cref{eq:2D_discr_DGSEM_with_GV_lin_vector_form} is maximum principle preserving.
\end{lemma}

\begin{proof}
This is a direct application of \cref{th:M-matrixCorr} to show that the linear system \cref{eq:2D_discr_DGSEM_with_GV_lin_vector_form} is defined from an M-matrix. Positivity preservation is then enough to get maximum principle preservation. We rewrite \cref{eq:2D_discr_DGSEM_with_GV_lin_vector_form} componentwise as

\begin{equation*}
a_{kl}U_{ij}^{kl,n+1} + \sum_{m=0, m\ne k}^{p} b_{klm}U_{ij}^{ml,n+1}  + \sum_{m=0, m\ne l}^{p} c_{klm}U_{ij}^{km,n+1} - \lambda_{x_i}\frac{\omega_l}{2}\delta_{k0}U_{(i-1)j}^{pl,n+1} - \lambda_{y_j}\frac{\omega_k}{2}\delta_{l0}U_{i(j-1)}^{kp,n+1} = \frac{\omega_k \omega_l}{4} U_{ij}^{kl,n}
\end{equation*}

\noindent where the diagonal coefficients read

\begin{equation*}
 a_{kl} = \frac{\omega_k \omega_l}{4} - \lambda_{x_i}\frac{\omega_l}{2}\big(\omega_kD_{kk}-\delta_{kp}\big) - \lambda_{y_j}\frac{\omega_k}{2}\big(\omega_lD_{ll}-\delta_{lp}\big) + \frac{\omega_k \omega_l}{2}d_{ij}\Big(\lambda_{x_i}\sum_{m=0, m\ne k}^{p}\frac{\omega_m}{2}+\lambda_{y_j}\sum_{m=0, m\ne l}^{p}\frac{\omega_m}{2}\Big)
\end{equation*}

\noindent and are positive since $-(\omega_kD_{kk}-\delta_{kp})=\tfrac{1}{2}(\delta_{kp}+\delta_{k0})$ and $-(\omega_lD_{ll}-\delta_{lp})=\tfrac{1}{2}(\delta_{lp}+\delta_{l0})$ from the SBP property \cref{eq:SBP}. Likewise, $b_{klm}=-\lambda_{x_i}\tfrac{\omega_l}{2}\omega_m(D_{mk}+\tfrac{\omega_k}{2}d_{ij})$ and $c_{klm}=-\lambda_{y_j}\tfrac{\omega_k}{2}\omega_m(D_{ml}+\tfrac{\omega_l}{2}d_{ij})$ are nonpositive under \cref{eq:cond_on_dij_for_M_matrix}.

Finally, strict diagonal dominance reads $a_{kl}>-\sum_{m\neq k}b_{klm}-\sum_{m\neq l}c_{klm}+\lambda_{x_i}\tfrac{\omega_l}{2}\delta_{k0}+\lambda_{y_j}\tfrac{\omega_k}{2}\delta_{l0}$ since $b_{klm}\leq0$ and $c_{klm}\leq0$. This reduces to

\begin{equation*}
 \frac{\omega_k \omega_l}{4} - \lambda_{x_i}\frac{\omega_l}{2}\big(\omega_kD_{kk}-\delta_{kp}\big)- \lambda_{y_j}\frac{\omega_k}{2}\big(\omega_lD_{ll}-\delta_{lp}\big) > \lambda_{x_i}\frac{\omega_l}{2}\sum_{m=0, m\ne k}^{p}\omega_mD_{mk} + \lambda_{y_j}\frac{\omega_k}{2}\sum_{m=0, m\ne l}^{p}\omega_mD_{ml} + \lambda_{x_i}\tfrac{\omega_l}{2}\delta_{k0} + \lambda_{y_j}\tfrac{\omega_k}{2}\delta_{l0},
\end{equation*}

\noindent where all the coefficients from the graph viscosity cancel each other out from \cref{eq:graph_viscosity_compnentwise} and have been removed. This can be rearranged into

\begin{equation*}
 \frac{\omega_k \omega_l}{4} > \lambda_{x_i}\frac{\omega_l}{2}\Big(\sum_{m=0}^p\omega_mD_{mk}-\delta_{kp}+\delta_{k0}\Big) + \lambda_{y_j}\frac{\omega_k}{2}\Big(\sum_{m=0}^p\omega_mD_{ml}-\delta_{lp}+\delta_{l0}\Big) \overset{\cref{eq:INT}}{=} 0,
\end{equation*}

\noindent  which is always satisfied and concludes the proof. \qed


\end{proof}

The modified DGSEM scheme with graph viscosity therefore satisfies the maximum principle for large enough $d_{ij}$ values. Table \ref{tab:lower_bounds_d} gives the minimum $d_{ij}$ values \cref{eq:cond_on_dij_for_M_matrix} in \cref{th:Mmat_condition_on_GV} guaranteeing a maximum principle. Likewise, the diagonal blocks in \cref{eq:2D_discr_DGSEM_with_GV_lin_vector_form} are now strictly diagonally dominant and hence invertible.

\begin{table}
  \caption{Lower bounds of the coefficient $d_{ij}>d_{min}$ for \cref{eq:2D_discr_DGSEM_with_GV_lin_vector_form} to be maximum-principle preserving}
  \begin{center}
  \begin{tabular}{lllllll}
    \hline
		p & 1 & 2 & 3 & 4 & 5 & 6 \\
    \hline
		$d_{min}$ & 1 & 3 & 3(1 + $\sqrt{5}$) & 24.8 & 53.6 & 102.6 \\
    \hline
  \end{tabular}
  \label{tab:lower_bounds_d}
  \end{center}
\end{table}

The scheme is however first order in space when $d_{ij}>0$ and is not used in practice. In the following, it is combined with the high-order scheme within the FCT limiter framework to keep high-order accuracy.

\subsection{Flux-corrected transport limiter}\label{sec:FCT_limiter}

Following \cite{Guermond_IDP_NS_2021}, the Flux-corrected transport (FCT) limiter \cite{BORIS_Book_FCT_73,zalesak1979fully} can be applied to guarantee a maximum principle by combining the high-order (HO) DGSEM scheme \cref{eq:2D_discr_DGSEM_lin_vector_form} and the low-order (LO) modified DGSEM scheme \cref{eq:2D_discr_DGSEM_with_GV_lin_vector_form} with graph viscosity. We here propose to use the FCT to guarantee a maximum principle on the cell-averaged solution \cref{eq:DGSEM_cell_averaged_sol_2d}, the maximum principle on all DOFs within the elements being ensured through the use of the linear scaling limiter (see \cref{sec:pos_limiter}) as in one space dimension.

By ${u}_{h,LO}^{(n+1)}$ and ${u}_{h,HO}^{(n+1)}$ we denote the solutions to the LO and HO schemes, respectively. Both are solutions to the cell-averaged scheme \cref{eq:DGSEM_cell_averaged_2d}. Subtracting the cell-averaged for the LO solution from the one for the HO solution gives

\begin{equation*}
\begin{aligned}
    \avg{{u}_{h,HO}^{(n+1)}}_{ij} - \avg{{u}_{h,LO}^{(n+1)}}_{ij}   \overset{\cref{eq:DGSEM_cell_averaged_2d}}{=} & \lambda_{x_i}\sum_{l=0}^{p} \frac{\omega_l}{2} \Big(U_{(i-1)j,HO}^{pl,n+1} - U_{ij,HO}^{pl,n+1} + U_{ij,LO}^{pl,n+1} - U_{(i-1)j,LO}^{pl,n+1}\Big) \\
 & + \lambda_{y_j}\sum_{k=0}^{p} \frac{\omega_k}{2} \Big(U_{i(j-1),HO}^{kp,n+1} - U_{ij,HO}^{kp,n+1} + U_{ij,LO}^{kp,n+1} - U_{i(j-1),LO}^{kp,n+1}\Big) \\
 = & \lambda_{x_i}\sum_{l=0}^{p} \frac{\omega_l}{2} \Big(U_{(i-1)j,HO}^{pl,n+1}  - U_{(i-1)j,LO}^{pl,n+1}\Big) + \lambda_{x_i}\sum_{l=0}^{p} \frac{\omega_l}{2} \Big(- U_{ij,HO}^{pl,n+1}  + U_{ij,LO}^{pl,n+1}\Big) \\
 & + \lambda_{y_j} \sum_{k=0}^{p} \frac{\omega_k }{2}\Big(U_{i(j-1),HO}^{kp,n+1} - U_{i(j-1),LO}^{kp,n+1}\Big) + \lambda_{y_j} \sum_{k=0}^{p} \frac{\omega_k }{2}\Big(- U_{ij,HO}^{kp,n+1} + U_{ij,LO}^{kp,n+1}\Big) \\
 \eqqcolon & A_{ij}^{(i-1)j} + A_{ij}^{(i+1)j} + A_{ij}^{i(j-1)} + A_{ij}^{i(j+1)} \\
 = & \sum_{(r,s)\in \mathcal{S}(i,j)} A_{ij}^{rs}
\end{aligned}
\end{equation*}

\noindent with $\mathcal{S}(i,j) = \{(i-1,j);(i+1,j);(i,j-1);(i,j+1)\}$. Note that we have $A_{ij}^{(i-1)j}=-A_{(i-1)j}^{ij}$ and $A_{ij}^{i(j-1)}=-A_{i(j-1)}^{ij}$. Again following \cite[Sec.~ 5.3]{Guermond_IDP_NS_2021}, we introduce the limiter coefficients defined by

\begin{subequations}\label{eq:FCT_limiter_params}
\begin{align}
P_{ij}^{-} & = \sum_{(r,s)\in \mathcal{S}(i,j)} \min\big(A_{ij}^{rs},0\big) \leq 0, & \quad Q_{ij}^{-} & = m - \avg{u_{LO}^{(n+1)}}_{ij} \leq 0, & l_{ij}^{-} & = \min\Big(1,\frac{Q_{ij}^{-}}{P_{ij}^{-}}\Big) \in [0,1], &  \\
P_{ij}^{+} & = \sum_{(r,s)\in \mathcal{S}(i,j)} \max\big(A_{ij}^{rs},0\big) \geq 0, & Q_{ij}^{+} & = M - \avg{u_{LO}^{(n+1)}}_{ij} \geq 0, & \quad l_{ij}^{+} & = \min\Big(1,\frac{Q_{ij}^{+}}{P_{ij}^{+}}\Big) \in [0,1],
\end{align}
\end{subequations}

\noindent where $m$ and $M$ are the lower and upper bounds in \cref{eq:PDE_max_principle} we want to impose to $\avg{{u}_{h}^{(n+1)}}_{ij}$. The new update of the mean value of the solution is now defined by:

\begin{equation}
    \avg{{u}_{h}^{(n+1)}}_{ij} - \avg{{u}_{h,LO}^{(n+1)}}_{ij}  = \sum_{(r,s)\in \mathcal{S}(i,j)} l_{ij}^{rs} A_{ij}^{rs}, \quad l_{ij}^{rs} = \left\{
    \begin{array}{ll}
        \min(l_{ij}^{-},l_{rs}^{+}) & \mbox{if } A_{ij}^{rs} < 0 \\
        \min(l_{rs}^{-},l_{ij}^{+}) & \mbox{otherwise.}
    \end{array}
\right.
\label{eq:Guermond_limiting_average}
\end{equation}

As a consequence, the cell-averaged solution satisfies the maximum principle (see \cite[Lemma~ 5.4]{Guermond_IDP_NS_2021}): using \cref{eq:Guermond_limiting_average} we get

\begin{align*}
 \avg{{u}_{h}^{(n+1)}}_{ij} - \avg{{u}_{h,LO}^{(n+1)}}_{ij}  \geq & \sum\limits_{(r,s)\in \mathcal{S}(i,j)} l_{ij}^{rs} \min(A_{ij}^{rs},0) = \min(l_{ij}^{-},l_{rs}^{+})P_{ij}^{-} \geq l_{ij}^{-}P_{ij}^{-} \geq Q_{ij}^{-} =  m - \avg{u_{LO}^{(n+1)}}_{ij}, \\
 \avg{{u}_{h}^{(n+1)}}_{ij} - \avg{{u}_{h,LO}^{(n+1)}}_{ij} \leq & \sum\limits_{(r,s)\in \mathcal{S}(i,j)} l_{ij}^{rs} \max(A_{ij}^{rs},0) = \min(l_{rs}^{-},l_{ij}^{+}) P_{ij}^{+} \leq l_{ij}^{+} P_{ij}^{+} \leq  Q_{ij}^{+} = M - \avg{u_{LO}^{(n+1)}}_{ij},
\end{align*}

\noindent since $Q_{ij}^{-}\leq l_{ij}^{-}P_{ij}^{-}\leq 0$ and  $0\leq l_{ij}^{+}P_{ij}^{+}\leq Q_{ij}^{+}$ by definition \cref{eq:FCT_limiter_params}.

Likewise, by \cref{eq:Guermond_limiting_average} we have $l_{ij}^{rs}=l_{rs}^{ij}$ for $(r,s)\in\mathcal{S}(i,j)$, thus ensuring conservation of the method:

\begin{equation*}
 \sum_{ij}\avg{{u}_{h}^{(n+1)}}_{ij} = \sum_{ij}\avg{{u}_{h,HO}^{(n+1)}}_{ij} = \sum_{ij}\avg{{u}_{h,LO}^{(n+1)}}_{ij} =  \sum_{ij}\avg{{u}_{h}^{(n)}}_{ij},
\end{equation*}

\noindent for periodic boundary conditions or compactly supported solutions.


From \cref{eq:Guermond_limiting_average}, the limiter is only applied at the interfaces and the DOFs can be evaluated explicitly from $u_{h,LO}^{(n+1)}$ and $u_{h,HO}^{(n+1)}$ through

\begin{align*}
    \frac{\omega_k\omega_l}{4}\big(U_{ij}^{kl,n+1} - U_{ij,HO}^{kl,n+1}\big) & = \delta_{kp} \frac{\omega_l\lambda_{x_i}}{2}\big(1-l_{ij}^{(i+1)j}\big)\Big(U_{ij,HO}^{pl,n+1}-U_{ij,LO}^{pl,n+1}\Big)  - \delta_{k0} \frac{\omega_l\lambda_{x_i}}{2}\big(1-l_{ij}^{(i-1)j}\big)\Big(U_{(i-1)j,HO}^{pl,n+1}-U_{(i-1)j,LO}^{pl,n+1}\Big) \\
    & + \delta_{lp} \frac{\omega_k\lambda_{y_j}}{2}\big(1-l_{ij}^{i(j+1)}\big)\Big(U_{ij,HO}^{kp,n+1}-U_{ij,LO}^{kp,n+1}\Big) - \delta_{l0} \frac{\omega_k\lambda_{y_j}}{2}\big(1-l_{ij}^{i(j-1)}\big)\Big(U_{i(j-1),HO}^{kp,n+1}-U_{i(j-1),LO}^{kp,n+1}\Big).
\end{align*}

This limited scheme is conservative, satisfies the maximum principle for the cell-averaged solution, but requires to solve two linear systems for $u_{h,LO}^{(n+1)}$ and $u_{h,HO}^{(n+1)}$ at each time step. Let us stress that since we need to compute $u_{h,HO}^{(n+1)}$, we know easily if the limiter is required, that is if the maximum principle is violated for the cell-averaged solution in some cell of the mesh. If it is not violated, we set $u_{h}^{(n+1)}\equiv u_{h,HO}^{(n+1)}$ and do not need to compute $u_{h,LO}^{(n+1)}$, but only to apply the linear scaling limiter (see \cref{sec:pos_limiter}). The FCT limiter may hence be viewed as an a posteriori limiter which is applied when needed after the solution update in the same way as other a posteriori limiters, such as in the MOOD method \cite{clain_etal_MOOD_11}. Preserving the maximum principle on the cell-averaged solution is a weaker requirement than preserving it on every DOFs and should therefore be more likely to be respected. As a consequence, the present FCT limiter is expected to less modify the solution,  which is supported by the numerical experiments of \cref{sec:num_xp_2D}.

In the next section, we also propose efficient algorithms to solve these linear systems to mitigate the extra cost induced by the additional linear solution when $u_{h,LO}^{(n+1)}$ is required.

\subsection{Linear system solution}\label{sec:2D_linear_sys_solution}

Both linear systems without, \cref{eq:2D_discr_DGSEM_lin_vector_form}, and with graph viscosity, \cref{eq:2D_discr_DGSEM_with_GV_lin_vector_form}, result in a block linear system

\begin{equation}\label{eq:global_lin_sys_2D}
 \mathbb{A}_{2d}{\bf U}^{(n+1)} = \mathbb{M}_{2d}{\bf U}^{(n)}
\end{equation}

\noindent of size $N_xN_yN_p$ with blocks of size $N_p=(p+1)^2$ to be solved for ${\bf U}^{(n+1)}$ where ${\bf U}_{n_{kl}+(i-1)N_p+(j-1)N_xN_p}=U_{ij}^{kl}$ with $n_{kl}=1+k+l(p+1)$ and $\mathbb{M}_{2d}$ the global mass matrix. Considering the block structure of $\mathbb{A}_{2d}$ is important for efficiently solving \cref{eq:global_lin_sys_2D} and usually requires the inversion of the diagonal blocks as a main step. These blocks are dense and hence require algorithms of complexity ${\cal O}(N_p^3)$ for their inversion. We propose below algorithms based on the properties of the 1D schemes in \cref{sec:1D_linear_sys_solution} for their efficient inversion. A repository of these algorithms (equations \cref{eq:eigenmodes_matA}, \cref{eq:inverse_L2d}, \cref{eq:2D_diag_block_with_graph_visc} and \cref{algo:invert_L2d^v}) is available at \cite{FastMatrixInverse_github} and \cref{app:fast_inversion_diag_blocks} provides a description of the repository.

\subsubsection{1D diagonal blocks as building blocks of the 2D linear systems}\label{sec:1D_building_blocks}

Let us introduce the diagonalization in $\mathbb{C}$ of the matrix $\mathbfcal{L}$ in \cref{eq:1D_diag_blocks}:

\begin{equation}\label{eq:calL_from_psib}
\mathbfcal{L}={\bf R}\Psib{\bf R}^{-1},
\end{equation}

\noindent where the columns of ${\bf R}\in\mathbb{C}^{(p+1)\times(p+1)}$ are the right eigenvectors of $\mathbfcal{L}$ and $\Psib$ is the diagonal matrix of the corresponding $p+1$ eigenvalues. We therefore have

\begin{equation}\label{eq:L1D_from_psib}
{\bf L}_{1d}={\bf R}\Psib_\lambda{\bf R}^{-1}, \quad \Psib_\lambda={\bf I}-2\lambda\Psib,
\end{equation}

\noindent for the 1D diagonal blocks in \cref{eq:1D_diag_blocks}.

From \cref{eq:1D_diag_blocks}, eigenpairs $\psi$ and ${\bf r}=(r_0,\dots,r_p)^\top$, such that $\mathbfcal{L}{\bf r}=\psi{\bf r}$, satisfy $\sum_lD_{lk}r_l-\delta_{kp}\tfrac{r_p}{\omega_p}=\psi r_k$ and summing this relation over $0\leq k\leq p$ gives $-\tfrac{1}{\omega_p}r_p=\psi\sum_kr_k$ and for $\psi=0$ we would have $r_p=0$, hence ${\bf D}^\top{\bf r}=0$ so ${\bf r}=0$ since ${\bf D}^\top$ is of rank $p$. So we have $\psi\neq0$ and we can invert the above relation with \cref{eq:def_calD} to get ${\bf r}=-\tfrac{r_p}{\psi\omega_p}\big(\sum_{l=0}^p\psi^{-l}{\bf D}^l\big)^\top{\bf e}_p$ and the $p$th component with $r_p\neq0$ gives the $\psi$ as the roots of the polynomial

\begin{subequations}\label{eq:eigenmodes_matA}
\begin{equation}\label{eq:eigenvalues_matA}
 \omega_p\psi^{p+1}+\sum_{l=0}^p\psi^{p-l}D_{pp}^{(l)}=0,
\end{equation}

\noindent and the eigenvector associated to any eigenvalue $\psi$ may be explicitly computed from

\begin{equation}\label{eq:eigenvectors_matA}
 r_k = -\frac{1}{\omega_p}\sum_{l=0}^p\psi^{-l-1}D_{pk}^{(l)} \quad \forall 0\leq k\leq p-1, \quad r_p=1.
\end{equation}
\end{subequations}

\subsubsection{Solution of the HO scheme \cref{eq:2D_discr_DGSEM_lin_vector_form}}

Setting $\lambda=\lambda_{x_i}+\lambda_{y_j}>0$, we rewrite the scheme \cref{eq:2D_discr_DGSEM_lin_vector_form} without graph viscosity as

\begin{align*} 
{\bf L}_{2d}({\bf M}\otimes{\bf M}){\bf U}_{ij}^{n+1} - \lambda_{x_i}({\bf M}\otimes{\bf e}_{0}{\bf e}_{p}^\top){\bf U}_{(i-1)j}^{n+1} - \lambda_{y_j}({\bf e}_{0}{\bf e}_{p}^\top\otimes{\bf M}){\bf U}_{i(j-1)}^{n+1} = ({\bf M}\otimes{\bf M}){\bf U}_{ij}^{n},
\end{align*}

\noindent where the first matrix in the diagonal blocks may be written as follows from the definition of ${\bf L}_{1d}$ in \cref{eq:L1D_from_psib}:
\begin{subequations}\label{eq:L2d_mat_def}
\begin{align}
 {\bf L}_{2d} \coloneqq& \frac{\lambda_{x_i}}{\lambda}{\bf I}\otimes{\bf L}_{1d} + \frac{\lambda_{y_j}}{\lambda}{\bf L}_{1d}\otimes{\bf I} = ({\bf R}\otimes{\bf R})\Psib_{2d}({\bf R}\otimes{\bf R})^{-1} \\
 \Psib_{2d} =& \frac{\lambda_{x_i}}{\lambda}{\bf I}\otimes\Psib_\lambda + \frac{\lambda_{y_j}}{\lambda}\Psib_\lambda\otimes{\bf I}, 
\end{align}
\end{subequations}

\noindent with ${\bf I}$ the identity matrix in $\mathbb{R}^{p+1}$ and ${\bf L}_{1d}$ the 1D operator defined in \cref{eq:1D_diag_blocks}. The diagonal matrix $\Psib_{2d}$ has $1-2(\lambda_{x_i}\psi_l+\lambda_{y_j}\psi_k)\neq0$ as $n_{kl}$th component. Hence the inverse of the diagonal blocks in \cref{eq:2D_discr_DGSEM_lin_vector_form} has an explicit expression

\begin{align}
 ({\bf M}\otimes{\bf M})^{-1}{\bf L}_{2d}^{-1} &= \left(({\bf M}^{-1}{\bf R})\otimes({\bf M}^{-1}{\bf R})\right)\Psib_{2d}^{-1}\left({\bf R}\otimes{\bf R}\right)^{-1}, \label{eq:inverse_L2d}\\
 \Psib_{2d}^{-1} &= \diag\left(\frac{1}{1-2(\lambda_{x_i}\psi_k+\lambda_{y_j}\psi_l)}:\;1\leq n_{kl}=1+k+l(p+1)\leq N_p\right). \nonumber
\end{align}

Note that $\mathbfcal{L}$ in \cref{eq:1D_diag_blocks} depends only on the approximation order of the scheme $p$, not on the $\lambda_{x_i}$ and $\lambda_{y_j}$, so the matrices ${\bf R}$, ${\bf R}^{-1}$, ${\bf M}^{-1}{\bf R}$, $\Psib$, $\Psib_{2d}$, etc. may be computed once from \cref{eq:eigenmodes_matA} at the beginning of the computation.

\subsubsection{Solution of the LO scheme \texorpdfstring{\cref{eq:2D_discr_DGSEM_with_GV_lin_vector_form}}{}}

Including the graph viscosity \cref{eq:2D_graph_viscosity} into \cref{eq:2D_discr_DGSEM_lin_vector_form} modifies the diagonal blocks of the linear system  and we now need to solve

\begin{align*} 
{\bf L}_{2d}^v({\bf M}\otimes{\bf M}){\bf U}_{ij}^{n+1} &- \lambda_{x_i}({\bf M}\otimes{\bf e}_{0}{\bf e}_{p}^\top){\bf U}_{(i-1)j}^{n+1} - \lambda_{y_j}({\bf e}_{0}{\bf e}_{p}^\top\otimes{\bf M}){\bf U}_{i(j-1)}^{n+1} = ({\bf M}\otimes{\bf M}){\bf U}_{ij}^{n},
\end{align*}

\noindent with

\begin{equation}\label{eq:2D_diag_block_with_graph_visc}
 {\bf L}_{2d}^v = {\bf L}_{2d}^0 - {\bf U}_v{\bf V}_v^\top, 
\end{equation}

\noindent and

\begin{subequations}\label{eq:matL2d0_U_V}
\begin{equation}\label{eq:matL2d0_U_V-a}
 {\bf L}_{2d}^0 = {\bf L}_{2d} + 2d_{ij}\lambda{\bf I}\otimes{\bf I} = ({\bf R}\otimes{\bf R})\big(\Psib_{2d}+2 d_{ij}\lambda{\bf I}\otimes {\bf I}\big)({\bf R}\otimes{\bf R})^{-1}, 
\end{equation}

\begin{equation}\label{eq:matL2d0_U_V-b}
 {\bf U}_v = 2d_{ij}\big(\lambda_{x_i}{\bf I}\otimes\omegab,\lambda_{y_j}\omegab\otimes{\bf I}\big), \quad {\bf V}_v = \big({\bf I}\otimes{\bf 1},{\bf 1}\otimes{\bf I}\big),
\end{equation}
\end{subequations}

\noindent where ${\bf U}_v$ and ${\bf V}_v$ are matrices in $\mathbb{R}^{N_p\times(2p+2)}$. Although the diagonal blocks ${\bf L}_{2d}^v$ may be efficiently built from the proposed method \cref{eq:2D_diag_block_with_graph_visc} (i.e., the 1D operators in ${\bf L}_{2d}^0$ plus a low-rank product) and then inverted with a direct solver, we propose below an alternative algorithm for their inversion that is found to be more efficient for polynomial degree up to $p\leq6$ (see \cref{app:fast_inversion_diag_blocks}).
Indeed, the matrix ${\bf L}_{2d}^0$ in \cref{eq:2D_diag_block_with_graph_visc} is easily inverted from \cref{eq:matL2d0_U_V-a} since $\Psib_{2d}+2 d_{ij}\lambda{\bf I}\otimes{\bf I}$ is diagonal. Then, we invert ${\bf L}_{2d}^v$ by using the Woodbury identity:

\begin{align*}
 ({\bf L}_{2d}^v)^{-1} \overset{\cref{eq:2D_diag_block_with_graph_visc}}{=}& \Big({\bf I}\otimes{\bf I} - ({\bf L}_{2d}^0)^{-1}{\bf U}_v{\bf V}_v^\top \Big)^{-1}({\bf L}_{2d}^0)^{-1} \\
 =& \bigg({\bf I}\otimes{\bf I} + ({\bf L}_{2d}^0)^{-1}{\bf U}_v\Big({\bf I}_{2p+2} + {\bf V}_v^\top({\bf L}_{2d}^0)^{-1}{\bf U}_v + \big({\bf V}_v^\top({\bf L}_{2d}^0)^{-1}{\bf U}_v\big)^2 + \dots\Big){\bf V}_v^\top \bigg)\,({\bf L}_{2d}^0)^{-1} \\
 =& \bigg({\bf I}\otimes{\bf I} + ({\bf L}_{2d}^0)^{-1}{\bf U}_v\Big({\bf I}_{2p+2}-{\bf V}_v^\top ({\bf L}_{2d}^0)^{-1} {\bf U}_v\Big)^{-1}{\bf V}_v^\top \bigg)\,({\bf L}_{2d}^0)^{-1},
\end{align*}

\noindent so the diagonal blocks may be inverted with \cref{algo:invert_L2d^v}, where only step 3 requires the inversion of a linear system of lower size with dense algebra tools. Steps 1 and 2 scale with ${\cal O}(4N_p^2)$ FLOPs, while step 3 and 4 require ${\cal O}(4(p+1)N_p^2)$ and ${\cal O}(2(p+1)N_p)$ FLOPs, respectively.

\begin{algorithm}
\caption{Algorithm flowchart for solving the system ${\bf L}_{2d}^v({\bf M}\otimes{\bf M}){\bf x}={\bf b}$ with graph viscosity.}\label{algo:invert_L2d^v}
\begin{algorithmic}[1]
\STATE{solve ${\bf L}_{2d}^0{\bf y}={\bf b}$ for ${\bf y}\in\mathbb{R}^{N_p}$ using \cref{eq:matL2d0_U_V-a}}:

\begin{equation*}
 {\bf y} = ({\bf R}\otimes{\bf R}) \diag\left( \frac{1}{1+2\lambda d_{ij}-2(\lambda_{x_i}\psi_k+\lambda_{y_j}\psi_l)}: \, n_{kl}\coloneqq1\leq 1+k+l(p+1)\leq N_p \right)({\bf R}^{-1}\otimes{\bf R}^{-1}){\bf b};
\end{equation*}
\STATE{solve ${\bf L}_{2d}^0{\bf Z}={\bf U}_v$ for ${\bf Z}\in\mathbb{R}^{N_p\times(2p+2)}$ using \cref{eq:matL2d0_U_V}}:

\begin{equation*}
 {\bf Z} = 2d_{ij}({\bf R}\otimes{\bf R}) \diag\left( \frac{1}{1+2\lambda d_{ij}-2(\lambda_{x_i}\psi_k+\lambda_{y_j}\psi_l)}: \, 1\leq n_{kl}\leq N_p \right) \big(\lambda_{x_i}{\bf R}^{-1}\otimes({\bf R}^{-1}\omegab),\lambda_{y_j}({\bf R}^{-1}\omegab)\otimes{\bf R}^{-1}\big);
\end{equation*}

\STATE{solve $({\bf I}_{2p+2}-{\bf V}_v^\top{\bf Z}){\bf z}={\bf V}_v^\top{\bf y}$ for ${\bf z}\in\mathbb{R}^{2p+2}$;}
\STATE{set ${\bf x} = ({\bf M}^{-1}\otimes{\bf M}^{-1})({\bf y} + {\bf Z}{\bf z})$.}
\end{algorithmic}
\end{algorithm}

%
%
\section{Numerical experiments}\label{sec:num_xp}

In this section we present numerical experiments on problems in one and two space dimensions (\cref{sec:num_xp_1D,sec:num_xp_2D}) in order to illustrate the properties of the DGSEM considered in this work. The FCT limiter \cref{eq:Guermond_limiting_average} is applied in the 2D experiments only. A maximum principle holds for the cell-averaged solution, $m\leq\avg{u_h^{(n+1)}}\leq M$, in one space dimension and in two space dimensions with the FCT limiter. We then apply the linear scaling limiter from \cite{zhang_shu_10a} described in \cref{sec:pos_limiter} to enforce a maximum principle on all the DOFs within the cells.

\subsection{Linear scaling limiter}\label{sec:pos_limiter}

Assuming $\avg{u^{(n+1)}}_\kappa \in [m,M]$ in a cell $\kappa$ (either $\kappa_i$ in 1D, or $\kappa_{ij}$ in 2D), Zhang and Shu \cite{zhang_shu_10a} proposed to modify ${\bf U}_\kappa^{(n+1)}$, the vector of DOFs in $\kappa$, as follows:

\begin{equation}\label{eq:pos_limiter}
    \widetilde{\bf U}_\kappa^{(n+1)} =\theta_\kappa {\bf U}_\kappa^{(n+1)} + (1-\theta_\kappa) \avg{u_h^{(n+1)}}_\kappa{\bf 1},
    \quad \theta_\kappa =\min \left(\left|\frac{M-\avg{u_h^{(n+1)}}_\kappa}{\max{\bf U}_\kappa-\avg{u_h^{(n+1)}}_\kappa}\right|,\left|\frac{m-\avg{u_h^{(n+1)}}_\kappa}{\min{\bf U}_\kappa-\avg{u_h^{(n+1)}}_\kappa}\right|,1 \right),
\end{equation}

\noindent with ${\bf 1}=(1,1,\dots,1)^\top\in\mathbb{R}^{(p+1)^d}$, and $\max{\bf U}_\kappa$ (resp., $\min{\bf U}_\kappa$) is the maximum (resp., minimum) value of the DOFs in the vector ${\bf U}_\kappa^{k,n+1}$. This limiter does not affect the high-order of accuracy for smooth solutions and does not change the cell average of the solution thus keeping the method conservative \cite{zhang_shu_10a}.

\subsection{One space dimension}\label{sec:num_xp_1D}

\subsubsection{Space-time accuracy}

First, the accuracy in time of the scheme can be checked with a smooth initial condition on a series of uniform grids with $\lambda_{1\leq i\leq N_x}=\lambda$. \Cref{tab:errImpl} displays the error levels and associated numerical orders of convergence with and without the linear scaling limiter \cref{eq:pos_limiter}. An accuracy of order one in time is observed and the limiter does not affect the time accuracy of the method.

\begin{table}
    \caption{Linear transport problem with $c_x=1$ and the initial data $u_0(x)=sin(2\pi x)$: $L^{k\in\{2,\infty\}}$ error levels $\|u_h-u\|_{L^k(\Omega_h\times[0,1])}$ and associated orders of convergence ${\cal O}_k$ obtained for $p=2$, $\lambda=1$ and refining the mesh. The linear scaling limiter \cref{eq:pos_limiter} is applied or not.}
    \centering
    \begin{tabular}{llllllllll}
        \hline
         & \multicolumn{4}{l}{no limiter} && \multicolumn{4}{l}{linear scaling limiter} \\
         \cline{2-5} \cline{7-10}
         $N_x$ & $L^{2}$ error & ${\cal O}_2$  & $L^{\infty}$ error & ${\cal O}_\infty$ && $L^{2}$ error & ${\cal O}_2$  & $L^{\infty}$ error & ${\cal O}_\infty$ \\
        \hline
         $20$ & 0.2650 & - & 0.3747 & - && 0.2650  & - & 0.3747 & - \\
         $40$ & 0.1517 & 0.80 & 0.2145 & 0.80 && 0.1517  & 0.80 & 0.2145 & 0.80 \\
         $80$ & 8.133E-2 & 0.90 & 0.1150 & 0.90 && 8.133E-2 & 0.90 & 0.1150 & 0.90 \\
         $160$ & 4.212E-2 & 0.95 & 5.956E-2 & 0.95 && 4.212E-2 & 0.95 & 5.956E-2 & 0.95  \\
         $320$ & 2.143E-2 & 0.97 & 3.031E-2 & 0.97 && 2.143E-2 & 0.97 & 3.031E-2 & 0.97 \\
         $640$ & 1.081E-2 & 0.99 & 1.529E-2 & 0.99 && 1.081E-2 & 0.99 & 1.529E-2 & 0.99 \\
        \hline
    \end{tabular}
    \label{tab:errImpl}
\end{table}


Then, the spatial accuracy is checked by looking for steady-state solutions of the following problem with a geometric source term and an inflow boundary condition:

\begin{equation} \label{eq:steady_state_1D}
  \partial_t u + \partial_x u = 2\pi\cos (2\pi x) \quad \mbox{in } [0,1] \times [0,T], \quad u(0, \cdot) = 0 \quad \mbox{in } [0,T],
\end{equation}

\noindent whose exact solution reads $u(x)=\sin (2\pi x)$. We take $\lambda = 1$, start from  $u_0(x)=0$, and march in time until $\|u_h^{n+1}-u_h^{n}\|_2 \leq 10^{-14}$. The $p+1$ accuracy of DGSEM is observed in \cref{tab:errImpl_space}. As expected \cite{zhang_shu_10a,zhang2012maximum}, the limiter does not affect the accuracy of the method.

\begin{table}
    \centering
    \caption{Steady-state problem \cref{eq:steady_state_1D}: $L^{k\in\{2,\infty\}}$ error levels $\|u_h-u\|_{L^k(\Omega_h)}$ and associated orders of convergence ${\cal O}_k$ obtained with $\lambda=1$ when refining the mesh. The linear scaling limiter \cref{eq:pos_limiter} is applied or not.}
    \begin{tabular}{lllllllllll}
        \hline
         & & \multicolumn{4}{l}{no limiter} && \multicolumn{4}{l}{linear scaling limiter} \\
         \cline{3-6} \cline{8-11}
         $p$& $N_x$ & $L^{2}$ error & ${\cal O}_2$  & $L^{\infty}$ error & ${\cal O}_\infty$ && $L^{2}$ error & ${\cal O}_2$  & $L^{\infty}$ error & ${\cal O}_\infty$ \\
        \hline
        \multirow{4}{*}{1}
         & 20 & 2.092E-2 & - & 4.071E-2 & - && 1.900E-2  & - & 4.071E-2 & - \\
         & 40 & 5.239E-3 & 2.00 & 1.025E-2 & 1.99 && 4.997E-3  & 1.93 & 1.025E-2 & 1.99 \\
         & 80 & 1.310E-3 & 2.00 & 2.569E-3 & 2.00 && 1.280E-3 & 1.96 & 2.569E-3 & 2.00 \\
         & 160 & 3.276E-4 & 2.00 & 6.424E-4 & 2.00 && 3.239E-4 & 1.98 & 6.424E-4 & 2.00  \\
        \\
        \multirow{4}{*}{2}
         & 20 & 4.164E-4 & - & 1.274E-2 & - && 4.256E-4  & - & 1.274E-2 & - \\
         & 40 & 5.210E-5 & 3.00 & 1.609E-4 & 2.99 && 5.226E-5   & 3.03 & 1.609E-4  & 2.99 \\
         & 80 & 6.515E-6 & 3.00 & 2.017E-5 & 3.00 && 6.517E-6 & 3.00 & 2.017E-5& 3.00 \\
         & 160 & 8.144E-7 & 3.00 & 2.523E-6 & 3.00 && 8.144E-7 & 3.00 & 2.523E-6 & 3.00  \\
        \\
        \multirow{4}{*}{3}
         & 20 & 6.978E-6 & - & 2.669E-5 & - && 6.978E-6  & - & 2.669E-5  & - \\
         & 40 & 4.365E-7 & 4.00 & 1.685E-6 & 3.99 && 4.365E-7  & 4.00 & 1.685E-6  & 3.99 \\
         & 80 & 2.729E-8 & 4.00 & 1.056E-7 & 4.00 && 2.729E-8 & 4.00 & 1.056E-7 & 4.00 \\
         & 160 & 1.706E-9 & 4.00 & 6.604E-9 & 4.00 && 1.706E-9 & 4.00 & 6.604E-9 & 4.00  \\
        \\
        \multirow{4}{*}{4}
         & 20 & 1.008E-07 & - & 4.493E-07 & - && 1.008E-07  & - & 4.493E-07  & - \\
         & 40 & 3.153E-09 & 5.00 & 1.418E-08 & 4.99 && 3.153E-09  & 5.00 & 1.418E-08  & 4.99 \\
         & 80 & 9.854E-11 & 5.00 & 4.444E-10 & 5.00 && 9.854E-11 & 5.00 & 4.444E-10 & 5.00 \\
         & 160 & 3.080E-12 & 5.00 & 1.392E-11 & 5.00 && 3.080E-12 & 5.00 & 1.392E-11 & 5.00  \\
       \\
       \multirow{4}{*}{5}
         & 20 & 1.253E-09 & - & 6.274E-09 & - && 2.237E-09  & - & 1.249E-08  & - \\
         & 40 & 1.959E-11 & 6.00 & 9.902E-11 & 5.99 && 2.851E-11  & 6.29 & 1.978E-10  & 5.98 \\
         & 80 & 3.061E-13 & 6.00 & 1.554E-12 & 5.99 && 3.826E-13 & 6.22 & 3.111E-12 & 5.99 \\
         & 160 & 1.182E-14 & 4.69 & 1.125E-13 & 3.79 && 1.334E-14 & 4.84 & 1.395E-13 & 4.48  \\
        \hline
    \end{tabular}
    \label{tab:errImpl_space}
\end{table}

\subsubsection{Maximum-principle preservation}

We now compare experiments with the theoretical bounds on the time to space steps ratio $\lambda$ indicated in \cref{tab:lower_bounds_CFL} for the DGSEM scheme to be maximum-principle preserving. We use a discontinuous initial condition  composed of a Gaussian, a square pulse, a sharp triangle and a combination of semi-ellipses \cite{LiuQiu15} in the range $0\leq x\leq 1$:

\begin{equation} \label{eq:init_cond_ZS}
    u_0(x) =
    \left\{
    \begin{array}{ll}
        \frac{1}{6}\big(G(x, \beta, z-\delta) + 4G(x, \beta, z) + G(x, \beta, z+\delta)\big) & 0.04 \leq x \leq 0.24, \\
        1 & 0.28 \leq x \leq 0.48, \\
        1 - 10|x-0.62| & 0.52 \leq x \leq 0.72,  \\
        \frac{1}{6}\big(F(x, \alpha, a-\delta) + 4F(x, \alpha, a) + F(x, \alpha, a+\delta)\big) & 0.76 \leq x \leq 0.96, \\
        0 & \mbox{else}.
    \end{array}
\right.
\end{equation}

\noindent with $G(x, \beta, z)=e^{-\beta (x-z)^2}$, $F(x, \alpha, a)=\sqrt{\max(1-\alpha^2(x-a)^2, 0)}$, $a=0.86$, $z=0.14$, $\delta=0.005$, $\alpha=10$ and $\beta=\frac{\log(2)}{36\delta^2}$.

Table \ref{tab:maxPrcImp} displays the minimum and the maximum values of the cell average solution of \cref{eq:discr_DGSEM_lin} after a short physical time for different approximation orders and different values of $\lambda$. The results are in good agreement with the theoretical lower bounds in \cref{tab:lower_bounds_CFL} and for $p\geq2$ the maximum principle is seen to be violated on at least one mesh for the lowest value $\lambda=0.1$.

\begin{table}
    \centering
    \caption{Linear scalar equation with a discontinuous initial condition \cref{eq:init_cond_ZS}: evaluation of the maximum principle for the cell-averaged solution as proved in \cref{th:MP_preservation_CFL} and \cref{tab:lower_bounds_CFL} after a short physical time $t=0.01$. The solution should remain in the interval $[0,1]$. The linear scaling limiter \cref{eq:pos_limiter} is always applied.}
    \begin{tabular}{lllllll}
    \hline
     & & \multicolumn{2}{l}{$N_x = 100$} && \multicolumn{2}{l}{$N_x = 101$} \\
     \cline{3-4}  \cline{6-7}
     $p$  &  $\lambda$ & $\underset{1\leq i\leq N_x}{\min} \avg{u_h}_i$ & $\underset{1\leq i\leq N_x}{\max}\avg{u_h}_i$ && $\underset{1\leq i\leq N_x}{\min} \avg{u_h}_i$ & $\underset{1\leq i\leq N_x}{\max}\avg{u_h}_i$ \\
     \hline
       & 0.1 & 0.0 & 1.0 && 0.0 & 1.0 \\
     1 & 0.25 & 9.88E-9 & 1.0 && 3.68E-10 & 1.0  \\
       & 0.5 & 6.14E-6 & 1.0 && 9.40E-7 & 1.0 \\
     \\
       & 0.1 & -3.92E-3 & 1.0005 && -5.43E-3 & 1.005 \\
     2 & 0.25 & 0.0 & 1.0 && 0.0 & 1.0 \\
       & 0.5 & 5.35E-7 & 1.0 && 6.29E-8 & 1.0  \\
    \\
       & 0.1 & 0.0 & 1.0 && -5.52E-3 & 1.006 \\
     3 & 0.195137 & 0.0 & 1.0 && 0.0 & 1.0 \\
       & 0.5 & 6.29E-7 & 1.0 && 8.91E-8 & 1.0 \\
    \\
       & 0.1 & -4.00E-4 & 1.0002 && -3.58E-5 & 1.000007 \\
     4 & 0.151 & 0.0 & 1.0 && 0.0 & 1.0 \\
       & 0.5 & 1.32E-7 & 1.0 && 8.93E-8 & 1.0 \\
    \\
       & 0.1 & 0.0 & 1.0 && -2.45E-7 & 1.0005 \\
     5 & 0.147568 & 0.0 & 1.0 && 0.0 & 1.0 \\
       & 0.5 & 9.92E-8 & 1.0 && 9.06E-8 & 1.0 \\
    \\
       & 0.10 & -1.78E-05 & 1.000013 && -1.44E-4 & 1.000144 \\
     6 & 0.109977 & 0.0 & 1.0 && 0.0 & 1.0 \\
       & 0.5 & 0.0 & 1.0 && 0.0 & 1.0 \\
    \hline
    \end{tabular}
    \label{tab:maxPrcImp}
\end{table}

In \cref{tab:exp_lower_bounds_CFL} we experimentally evaluate the lower bound on $\lambda$ to guarantee a  maximum principle on the cell-averaged solution by using a bisection method from the same configuration as in \cref{tab:maxPrcImp}. We observe that the theoretical lower bound on $\lambda$ derived in \cref{th:MP_preservation_CFL} and \cref{tab:lower_bounds_CFL} is sharp and confirmed by the experimental observations, though it seems to slightly overestimate the experimental lower bound for $p=3$ and $p=5$. Let us recall that the condition $\lambda>\lambda_{min}$ in \cref{tab:lower_bounds_CFL} is a sufficient to obtain maximum principle preservation.

\begin{table}
  \begin{center}
  \caption{Experimental evaluation of the lower bounds of the time to space steps ratio $\lambda_{min}^{exp}$ such that $\lambda_{1\leq i\leq N_x}=\lambda\geq\lambda_{min}^{exp}$ ensures the maximum principle preservation for the cell-averaged solution in \cref{th:MP_preservation_CFL} and \cref{tab:lower_bounds_CFL}, while it doesn't for $\lambda\leq\lambda_{min}^{exp}-10^{-2}$ on meshes with $N_x=100$ and $N_x=101$ elements. We report the theoretical values from \cref{tab:lower_bounds_CFL} in the bottom line for the sake of comparison.}
  \begin{tabular}{lllllll}
    \hline
		p & 1 & 2 & 3 & 4 & 5 & 6\\
    \hline
		$\lambda_{min}^{exp}$ & $0$ & $0.25$ & $0.17$ & $0.16$ & $0.13$ & $0.11$\\
		$\lambda_{min}$ & $0$ & $0.25$ & $0.195137$ & $0.150346$ & $0.147568$ & $0.109977$\\
    \hline
  \end{tabular}
  \label{tab:exp_lower_bounds_CFL}
  \end{center}
\end{table}

\subsubsection{Linear advection-reaction with source}\label{sec:adv_reac_pb_1D}

We finally consider a linear advection-reaction problem with a geometric source term:

\begin{equation}\label{eq:adv_reaction_source_pb}
 \partial_tu + c_x\partial_xu + \beta u = s(x) \quad \text{in } \Omega\times(0,\infty), \quad u(x,0)=u_0(x) \quad \text{in }\Omega.
\end{equation}

\noindent with $\beta\geq0$ and $s(\cdot)\geq0$. Providing nonnegative initial and boundary data are imposed, the solution remains nonnegative for all time.

We here adapt the problem representative of the radiative transfer equations from \cite[Ex.~6.2]{Xu_Shu_cons_PP_2023} with $c_x=1$, $\beta=6000$, $s(x)=\beta(\tfrac{1}{9}\cos^4(2\pi x)+\epsilon)-\tfrac{4}{9}\cos^3(2\pi x)\sin(2\pi x)$ on $\Omega=[0,3]$, $\epsilon=10^{-14}$, and an inflow boundary condition $u(0,t)=\tfrac{1}{9}+\epsilon$ on $\Omega=[0,\pi]$. This problem has a steady-state smooth solution, but with low positive values and large oscillations: $u(x)=\tfrac{1}{9}\cos^4(2\pi x)+\epsilon\geq\epsilon$ (see \cref{fig:1D_adv_reac_pb}).

We again set $\lambda=1$ and iterate up to convergence $\|u_h^{n+1}-u_h^{n}\|_2 \leq 10^{-14}$. Table \ref{tab:err_1D_adv_reac_pb} displays the error levels obtained for different approximation orders and mesh refinements when applying the scaling limiter \cref{eq:pos_limiter} or not, together with the evaluation of the lowest value of the DGSEM solution. The limiter keeps the high-order accuracy of the DGSEM, while it successfully preserves positivity of the solution thus confirming that the DGSEM preserves positivity of the cell-averaged solution before the application of the limiter. We however observe a suboptimal $p$th order of accuracy with or without the limiter which was also reported in preceding experiments \cite{chen_shu_17} and can be attributed to the low accuracy of the Gauss-Lobatto quadrature rules applied to the nonlinear geometric source term compared to other quadrature rules \cite{cockburn-shu90} (see also \cite[Remark~3.1]{chen_shu_17}).

\begin{figure}
    \centering
    \includegraphics[width=8cm]{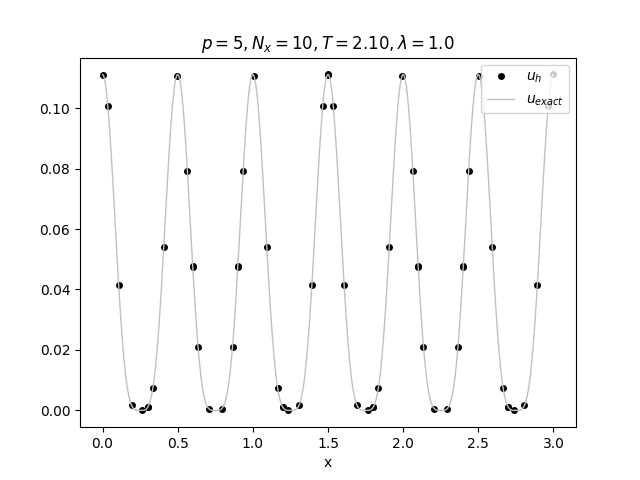}
    \caption{Advection-reaction equation with source \cref{eq:adv_reaction_source_pb}: steady-state DGSEM solution for $p=5$, $N_x=10$ and the linear scaling limiter \cref{eq:pos_limiter}. The solution is plotted at quadrature points and $T$ refers to the pseudo time required to converge the solution, i.e., $\|u_h^{n+1}-u_h^{n}\|_2 \leq 10^{-14}$.}
    \label{fig:1D_adv_reac_pb}
\end{figure}

\begin{table}
    \centering
    \caption{Advection-reaction problem with source \cref{eq:adv_reaction_source_pb}: $L^{k\in\{2,\infty\}}$ error levels $\|u_h-u\|_{L^k(\Omega_h)}$ and associated orders of convergence ${\cal O}_k$ obtained with $\lambda=1$ when refining the mesh. The solution should remain in the interval $[0,\tfrac{1}{9}]$. Minimum value of DOFs $u_{h_{\min}}=\min(U_{1\leq i\leq N_x}^{0\leq k\leq p})$.  The linear scaling limiter \cref{eq:pos_limiter} is applied or not.}
    \begin{tabular}{lllllllllllll}
        \hline
         & & \multicolumn{5}{l}{no limiter} && \multicolumn{5}{l}{linear scaling limiter} \\
         \cline{3-7} \cline{9-13}
         $p$& $N_x$ & $u_{h_{\min}}$ & $L^{2}$ error & ${\cal O}_2$ & $L^{\infty}$ error & ${\cal O}_\infty$ && $u_{h_{\min}}$ & $L^{2}$ error & ${\cal O}_2$  & $L^{\infty}$ error & ${\cal O}_\infty$ \\
        \hline
         & 10 & 9.39E-04 & 1.89E-04 & -- & 1.70E-04 & -- && 9.39E-04 & 1.96E-04 & -- & 1.70E-04 & --\\
         & 20 & -5.28E-05 & 1.50E-04 & 0.33 & 1.71E-04 &-0.01 && 9.99E-15 & 1.49E-04 & 0.39 & 1.71E-04 &-0.01\\
       1 & 40 & -1.04E-05 & 8.90E-05 & 0.76 & 1.02E-04 & 0.74 && 9.99E-15 & 8.42E-05 & 0.83 & 1.02E-04 & 0.74\\
         & 80 & -1.46E-06 & 4.63E-05 & 0.94 & 5.35E-05 & 0.94 && 1.00E-14 & 4.43E-05 & 0.93 & 5.35E-05 & 0.94\\
         & 160 & -3.01E-07 & 2.32E-05 & 1.00 & 2.70E-05 & 0.98 && 1.00E-14 & 2.26E-05 & 0.97 & 2.68E-05 & 1.00\\
        \\
         & 10 & -4.33E-08 & 1.35E-04 & -- & 1.60E-04 & -- && 9.99E-15 & 1.26E-04 & -- & 1.60E-04 & --\\
         & 20 & -3.17E-05 & 5.52E-05 & 1.29 & 7.38E-05 & 1.12 && 1.00E-14 & 6.78E-05 & 0.89 & 1.57E-04 & 0.03\\
       2 & 40 & -7.44E-06 & 1.58E-05 & 1.80 & 2.24E-05 & 1.72 && 1.00E-14 & 1.81E-05 & 1.91 & 4.35E-05 & 1.86\\
         & 80 & -1.05E-06 & 4.11E-06 & 1.95 & 5.90E-06 & 1.93 && 1.00E-14 & 4.21E-06 & 2.10 & 6.51E-06 & 2.74\\
         & 160 & -1.31E-07 & 1.03E-06 & 1.99 & 1.53E-06 & 1.94 && 1.00E-14 & 1.04E-06 & 2.02 & 1.53E-06 & 2.09\\
        \\
         & 10 & 1.62E-04 & 7.86E-05 & -- & 1.27E-04 & -- && 1.62E-04 & 8.90E-05 & -- & 1.27E-04 & --\\
         & 20 & -5.33E-06 & 1.60E-05 & 2.29 & 2.71E-05 & 2.22 && 9.99E-15 & 1.71E-05 & 2.38 & 3.66E-05 & 1.79\\
       3 & 40 & -1.46E-06 & 2.25E-06 & 2.83 & 3.84E-06 & 2.82 && 1.00E-14 & 2.71E-06 & 2.66 & 7.75E-06 & 2.24\\
         & 80 & -2.59E-07 & 2.86E-07 & 2.98 & 4.81E-07 & 3.00 && 1.00E-14 & 3.30E-07 & 3.04 & 1.07E-06 & 2.85\\
         & 160 & -3.36E-08 & 3.52E-08 & 3.02 & 6.22E-08 & 2.95 && 1.00E-14 & 3.81E-08 & 3.11 & 1.35E-07 & 2.99\\
        \\
         & 10 & -1.97E-05 & 3.90E-05 & -- & 6.46E-05 & -- && 9.99E-15 & 4.35E-05 & -- & 6.54E-05 & --\\
         & 20 & -5.26E-06 & 3.72E-06 & 3.39 & 6.56E-06 & 3.30 && 1.00E-14 & 5.68E-06 & 2.94 & 2.05E-05 & 1.67\\
       4 & 40 & -2.65E-07 & 2.58E-07 & 3.85 & 4.55E-07 & 3.85 && 1.00E-14 & 3.08E-07 & 4.21 & 1.33E-06 & 3.94\\
         & 80 & -8.99E-09 & 1.66E-08 & 3.96 & 3.11E-08 & 3.87 && 1.00E-14 & 1.71E-08 & 4.17 & 4.76E-08 & 4.81\\
         & 160 & -2.73E-10 & 1.04E-09 & 3.99 & 2.08E-09 & 3.90 && 1.00E-14 & 1.04E-09 & 4.03 & 2.08E-09 & 4.51\\
        \\
         & 10 & 1.95E-06 & 1.56E-05 & -- & 2.70E-05 & -- && 1.95E-06 & 1.92E-05 & -- & 3.48E-05 & --\\
         & 20 & -3.30E-07 & 7.12E-07 & 4.46 & 1.32E-06 & 4.35 && 9.99E-15 & 7.78E-07 & 4.63 & 1.32E-06 & 4.72\\
       5 & 40 & -2.62E-08 & 2.41E-08 & 4.88 & 4.58E-08 & 4.8 && 1.00E-14 & 3.05E-08 & 4.67 & 7.84E-08 & 4.08\\
         & 80 & -1.06E-09 & 7.54E-10 & 5.00 & 1.38E-09 & 5.04 && 1.00E-14 & 9.22E-10 & 5.05 & 3.10E-09 & 4.66\\
         & 160 & -3.14E-11 & 2.29E-11 & 5.04 & 4.63E-11 & 4.91 && 1.00E-14 & 2.56E-11 & 5.17 & 9.06E-11 & 5.10\\
            \hline
    \end{tabular}
    \label{tab:err_1D_adv_reac_pb}
\end{table}


%
\subsection{Two space dimensions}\label{sec:num_xp_2D}

We now focus on numerical tests in two space dimensions in the unit square using a Cartesian mesh with $N_x=N_y$ cells in the $x$ and $y$ directions respectively. For all the tests, we set $c_x = c_y = 1$. We here compare results obtained with the DGSEM scheme and without or with the FCT limiter:

\begin{description}
 \item{\textbf{no limiter}:} we solve \cref{eq:2D_discr_DGSEM_lin_vector_form} without graph viscosity for  $u_h^{(n+1)}$. We cannot apply the linear scaling limiter \cref{eq:pos_limiter} since the $\avg{u_h^{(n+1)}}_{ij}$ are not guaranteed to satisfy the maximum principle;
 \item{\textbf{FCT limiter}:} we first solve \cref{eq:2D_discr_DGSEM_lin_vector_form} without graph viscosity for $u_{h,HO}^{(n+1)}$ and check if the cell-averaged solution satisfies the maximum principle, and if so, we set $u_{h}^{(n+1)}\equiv u_{h,HO}^{(n+1)}$. If not, we solve \cref{eq:2D_discr_DGSEM_with_GV_lin_vector_form} with graph viscosity for $u_{h,LO}^{(n+1)}$ and apply the FCT limiter \cref{eq:Guermond_limiting_average} introduced in \cref{sec:FCT_limiter}. Finally, we apply the linear scaling limiter \cref{eq:pos_limiter} after the FCT limiter to preserve a maximum principle on all the DOFs in  $u_h^{(n+1)}$.
\end{description}

\subsubsection{Maximum-principle preservation}

We first evaluate both DGSEM schemes on an unsteady problem with a discontinuous initial condition, $u_0(x,y)=1$ if $|x-\frac{1}{4}|+|y-\frac{1}{4}|\leq0.15$ and $0$ else, and periodic boundary conditions. Table \ref{tab:maxPrcImp2D} gives the minimum and maximum values of the cell-averaged solution after one time step with $1\leq p\leq 5$ and different values of $\lambda_x=\lambda_y$. The maximum principle is not satisfied when using the DGSEM without limiter, except for $p=1$ and the smallest time step. In particular, the maximum principle is violated even for large $\lambda_x=\lambda_y$ in contrast to what is observed and proved in one space dimension. As expected, the FCT limiter successfully imposes a maximum principle on the cell-averaged solution, thus enabling a maximum principle through the use of the linear scaling limiter.


%
\begin{table}
    \centering
    \caption{Verification of the maximum principle for problem \cref{eq:hyp_2Dcons_laws} after one time step on a mesh with $N_x \times N_y = 20 \times 20$ elements, $\lambda_{x_i}=\lambda_{y_j}=\lambda$, and the discontinuous initial condition $u_0(x,y)=1_{|x-\frac{1}{4}|+|y-\frac{1}{4}|\leq0.15}$. The solution should remain in the interval $[0, 1]$.}
    \begin{tabular}{lllllll}
    \hline
     & & \multicolumn{2}{l}{no limiter} && \multicolumn{2}{l}{FCT limiter}\\
     \cline{3-4} \cline{6-7}
     p  &  $\lambda$ & $\underset{1\leq i,j\leq20}{\min} \avg{u_h}_{ij}$ & $\underset{1\leq i,j\leq20}{\max}\avg{u_h}_{ij}$ && $\underset{1\leq i,j\leq20}{\min} \avg{u_h}_{ij}$ & $\underset{1\leq i,j\leq20}{\max}\avg{u_h}_{ij}$ \\
     \hline
       & 0.05 & 0.0 & 1.0 && 0.0 & 1.0 \\
     1 & 1 & -9.45E-3 & 0.90 && 9.59E-08 & 0.90 \\
       & 5 & -9.45E-3  & 0.47 && 1.37E-02 & 0.49 \\
    \\
       & 0.05 & -6.76E-4 & 1.0002 && 0.0 & 1.0 \\
     2 & 1 & -6.76E-3 & 0.93 && 1.42E-07 & 0.90  \\
       & 5 & -6.60E-3 & 0.44 && 2.01E-3 & 0.41 \\
    \\
       & 0.05 & -4.85E-8 & 1.0 && 0.0 & 1.0 \\
     3 & 1 & -2.98E-3 & 0.92 && 5.40E-07 & 0.91 \\
       & 5 & -6.24E-4 & 0.44 && 3.22E-03 & 0.38 \\
    \\
       & 0.05 & -1.41E-4 & 1.0007 && 0.0 & 1.0 \\
     4 & 1 & -1.33E-4 & 0.92 && 5.97E-07 & 0.92 \\
       & 5 & -6.09E-5 & 0.44 && 3.25E-03 & 0.38 \\
    \\
       & 0.05 & -1.31E-3 & 1.0 && 0.0 & 1.0 \\
     5 & 1 & -8.80E-5 & 0.92 && 6.45E-07 & 0.92 \\
       & 5 & -1.10E-4 & 0.44 && 3.33E-03 & 0.38 \\
    \hline
    \end{tabular}
    \label{tab:maxPrcImp2D}
\end{table}

\subsubsection{Steady smooth solution}

We now consider a smooth steady-state solution of the problem $\partial_x u + \partial_y u = 0$ in $\Omega=[0,1]^2$ with inlet conditions $u(x,0)=\sin(2\pi x)$, $u(0,y)=-\sin(2\pi y)$ and outflow conditions at boundaries $x=1$ and $y=1$. The exact solution is $u(x,y)=\sin(2\pi(x-y))$. In practice, we look for a steady solution to the unsteady problem \cref{eq:hyp_cons_laws}. We take $\lambda_x=\lambda_y=5$, start from $u_h^{(0)}\equiv0$ and march in time until $\|u_h^{n+1}-u_h^{n}\|_2 \leq 10^{-14}$ with the DGSEM scheme with FCT limiter. Error levels are summarized in \cref{tab:err_2D_smooth_steady_pb} together with the minimum and maximum values of the cell-averaged solution. The FCT limiter keeps here the $p+1$ high-order accuracy in space of the DGSEM while it successfully preserves the maximum principle on the cell-averaged solution, and hence also on the DOFs through the linear scaling limiter.

\begin{table}
    \centering
    \caption{Smooth steady-state problem: $L^{k\in\{2,\infty\}}$ error levels $\|u_h-u\|_{L^k(\Omega_h)}$ and associated orders of convergence ${\cal O}_k$ for problem $\partial_x u + \partial_y u = 0$ with data $u(x,0)=\sin(2\pi x)$, $u(0,y)=-\sin(2\pi y)$ obtained with $\lambda_x=\lambda_y=5$ when refining the mesh and using the FCT limiter. The solution should remain in the interval $[-1, 1]$. Minimum and maximum values of the cell-averaged solution over the mesh: $\avg{u_h}_{\min/\max}=\min/\max(\avg{u_h}_{ij}:\,1\leq i\leq N_x, 1\leq j\leq N_y)$.}
    \begin{tabular}{llllllll}
        \hline
         $p$ & $N_x=N_y$ & $\avg{u_h}_{\min}$ & $\avg{u_h}_{\max}$ & $L^{2}$ error & ${\cal O}_2$ & $L^{\infty}$ error & ${\cal O}_\infty$ \\
        \hline
        \multirow{4}{*}{1}
         & 5  & -0.7803 & 0.7803 & $3.260E-01$ & -- & 6.805E-01 & --\\
         & 10 & -0.9313 & 0.9313 & 9.840E-02 & $1.73$ & 2.779E-01 & $1.29$\\
         & 20 & -0.9955 & 0.9955 & 2.431E-02 & $2.02$ & 6.341E-02 & $2.13$\\
         & 40 & -0.9991 & 0.9991 & 6.589E-03 & $1.88$ & 1.789E-02 & $1.83$\\
        \\
        \multirow{4}{*}{2}
         & 5  & -0.8293 & 0.8293 & 3.808E-02 & -- & 1.610E-01 & --\\
         & 10 & -0.9200 & 0.9200 & 4.770E-03 & $3.00$ & 1.348E-02 & $3.58$\\
         & 20 & -0.9917 & 0.9917 & 6.038E-04 & $2.98$ & 2.354E-03 & $2.52$\\
         & 40 & -0.9979 & 0.9979 & 7.377E-05 & $3.03$ & 2.084E-04 & $3.50$\\
        \\
        \multirow{4}{*}{3}
         & 5  & -0.8322 & 0.8322 & 2.511E-03 & -- & 8.746E-03 & --\\
         & 10 & -0.9201 & 0.9201 & 1.569E-04 & $4.00$ & 7.599E-04 & $3.52$\\
         & 20 & -0.9918 & 0.9918 & 1.074E-05 & $3.87$ & 7.432E-05 & $3.35$\\
         & 40 & -0.9979 & 0.9979 & 6.457E-07 & $4.06$ & 4.724E-06 & $3.98$\\
        \\
        \multirow{4}{*}{4}
         & 5  & -0.8323 & 0.8323 & 1.430E-04 & -- & 6.283E-04 & --\\
         & 10 & -0.9201 & 0.9201 & 4.545E-06 & $4.98$ & 1.880E-05 & $5.06$\\
         & 20 & -0.9918 & 0.9918 & 1.431E-07 & $4.99$ & 6.162E-07 & $4.93$\\
         & 40 & -0.9979 & 0.9979 & 4.461E-09 & $5.00$ & 1.950E-08 & $4.98$\\
        \\
        \multirow{4}{*}{5}
         & 5  & -0.8323 & 0.8323 & 7.131E-06 & -- & 3.774E-05 & --\\
         & 10 & -0.9201 & 0.9201 & 1.131E-07 & $5.98$ & 7.490E-07 & $5.65$\\
         & 20 & -0.9918 & 0.9918 & 4.074E-09 & $4.80$ & 6.652E-08 & $3.49$\\
         & 40 & -0.9979 & 0.9979 & 4.789E-11 & $6.41$ & 1.058E-09 & $5.97$\\
        \hline
    \end{tabular}
    \label{tab:err_2D_smooth_steady_pb}
\end{table}

\subsubsection{Steady discontinuous solution}

We now consider a discontinuous steady solution and consider $\partial_x u + \partial_y u = 0$ in $\Omega=[0,1]^2$, inlet conditions $u(x,0)=\cos(\pi x)$, $u(0,y)=-\cos(\pi y)$ and outflow conditions at boundaries $x=1$ and $y=1$. The exact solution is $u(x,y)=\text{sgn}(x-y)\cos(\pi(x-y))$, with $\text{sgn}$ the sign function, and is therefore discontinuous at $x=y$. Results are reported in \cref{tab:err_2D_disc_steady_pb} and \cref{fig:MP_2D_disc_steady_pb}. Here again, the FCT limiter is required to guarantee the maximum principle. In particular, the DGSEM without limiter violates the maximum principle for the cell-averaged solution which prevents the use of the linear scaling limiter.

\begin{table}
    \centering
    \caption{Discontinuous steady-state problem: verification of the maximum principle for problem $\partial_x u + \partial_y u = 0$ with data $u(x,0)=\cos(\pi x)$, $u(0,y)=-\cos(\pi y)$ obtained with $\lambda_x=\lambda_y=5$ without and with the FCT limiter, and with $N_x=N_y=N$. The solution should remain in the interval $[-1,1]$. Minimum and maximum values of the cell-averaged solution and DOFs over the mesh: $\avg{u_h}_{\min/\max}=\min/\max(\avg{u_h}_{ij}:\,1\leq i\leq N_x, 1\leq j\leq N_y)$ and $u_{h_{\min/\max}}=\min/\max(U_{1\leq i,j\leq N}^{0\leq k,l\leq p})$.}
    \begin{tabular}{lllllllllll}
        \hline
         & & \multicolumn{4}{l}{no limiter} && \multicolumn{4}{l}{FCT limiter} \\
         \cline{3-6}  \cline{8-11}
         $N$ & $p$ & $\avg{u_h}_{\min}$ & $\avg{u_h}_{\max}$ & $u_{h_{\min}}$ & $u_{h_{\max}}$ && $\avg{u_h}_{\min}$ & $\avg{u_h}_{\max}$ & $u_{h_{\min}}$ & $u_{h_{\max}}$ \\
        \hline
        \multirow{5}{*}{5}
         & 1 & -0.7518 & 0.7518 & -1.1363 & 1.1363 && -0.7512 & 0.7512 & -1.0000 & 1.0000 \\
         & 2 & -0.7820 & 0.7820 & -1.2634 & 1.2634 && -0.7820 & 0.7820 & -1.0000 & 1.0000 \\
         & 3 & -0.7972 & 0.7972 & -1.3364 & 1.3364 && -0.7827 & 0.7827 & -1.0000 & 1.0000 \\
         & 4 & -0.7832 & 0.7832 & -1.3633 & 1.3633 && -0.7828 & 0.7828 & -1.0000 & 1.0000 \\
         & 5 & -0.7828 & 0.7828 & -1.3764 & 1.3764 && -0.7828 & 0.7828 & -1.0000 & 1.0000 \\
        \\
        \multirow{5}{*}{20}
         & 1 & -1.0121 & 1.0121 & -1.2437 & 1.2437 && -0.9967 & 0.9967 & -1.0000 & 1.0000 \\
         & 2 & -1.0465 & 1.0465 & -1.2843 & 1.2843 && -0.9781 & 0.9781 & -1.0000 & 1.0000 \\
         & 3 & -1.0042 & 1.0042 & -1.3438 & 1.3438 && -0.9857 & 0.9857 & -1.0000 & 1.0000 \\
         & 4 & -0.9937 & 0.9937 & -1.3667 & 1.3667 && -0.9857 & 0.9857 & -1.0000 & 1.0000 \\
         & 5 & -0.9857 & 0.9857 & -1.3781 & 1.3781 && -0.9857 & 0.9857 & -1.0000 & 1.0000 \\
        \hline
    \end{tabular}
    \label{tab:err_2D_disc_steady_pb}
\end{table}

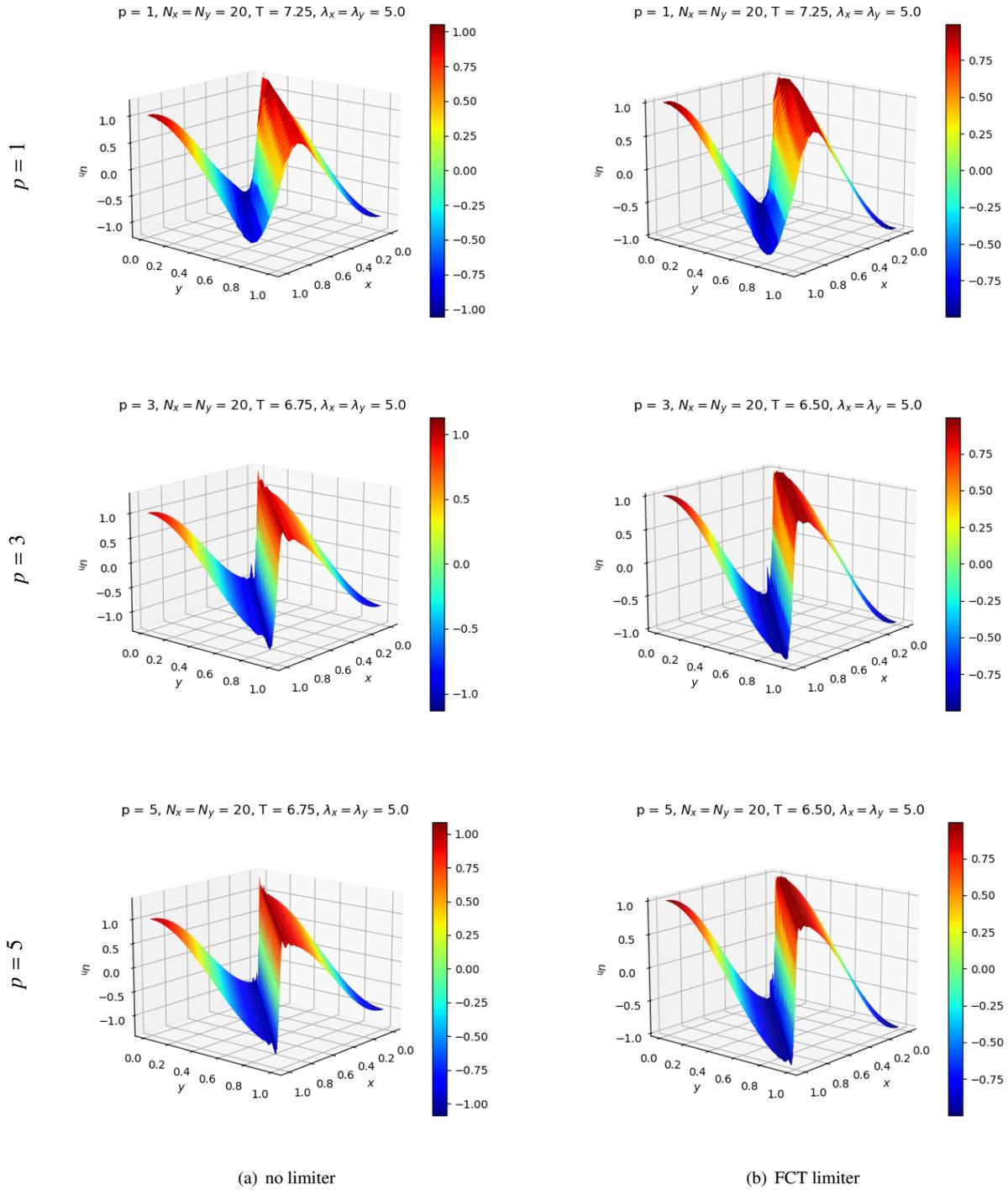
\begin{figure}
    \centering
    \subfigure{\begin{picture}(0,0)\put(-6,75){\rotatebox{90}{$p=1$}} \end{picture}}
    \includegraphics[width=8cm]{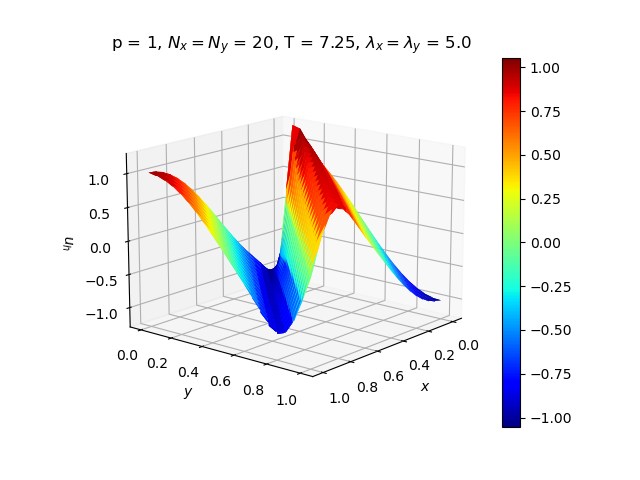}
    \includegraphics[width=8cm]{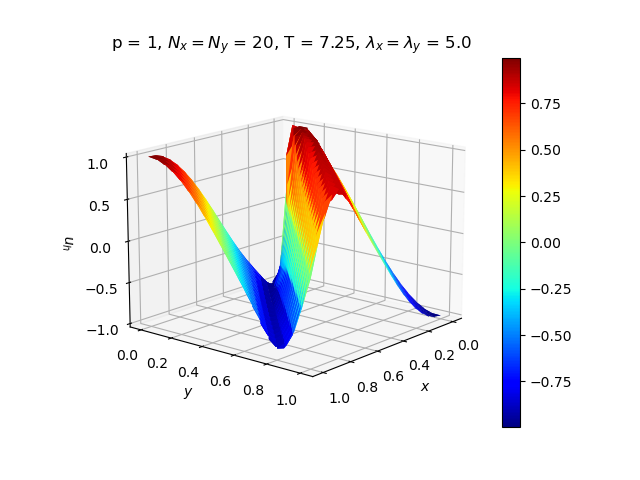}\\
    \subfigure{\begin{picture}(0,0)\put(-6,75){\rotatebox{90}{$p=3$}} \end{picture}}
    \includegraphics[width=8cm]{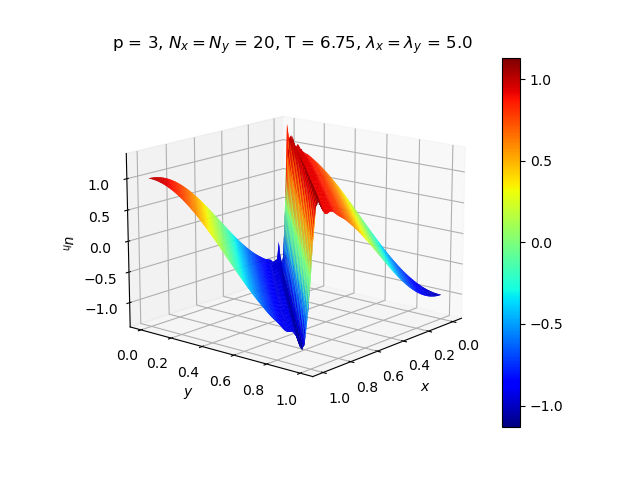}
    \includegraphics[width=8cm]{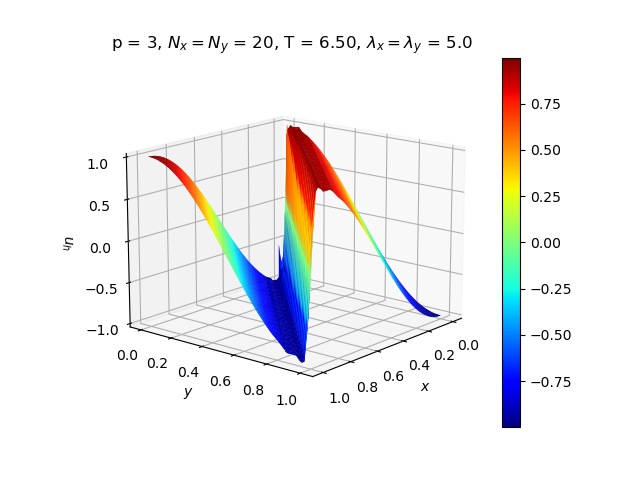}\\
    \subfigure{\begin{picture}(0,0)\put(-6,75){\rotatebox{90}{$p=5$}} \end{picture}}
    \setcounter{subfigure}{0}
    \subfigure[no limiter]{\includegraphics[width=8cm]{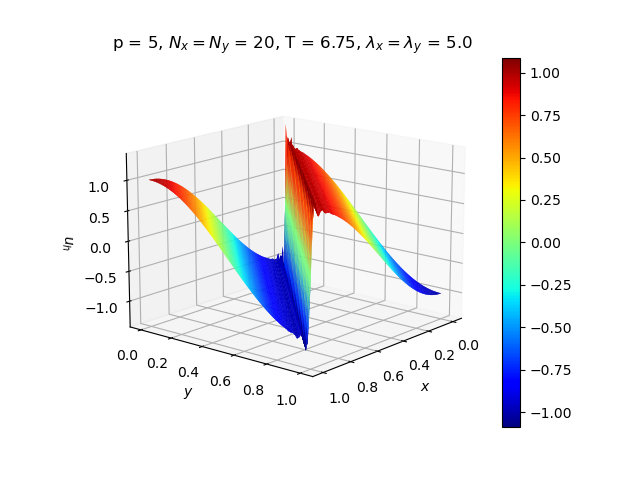}}
    \subfigure[FCT limiter]{\includegraphics[width=8cm]{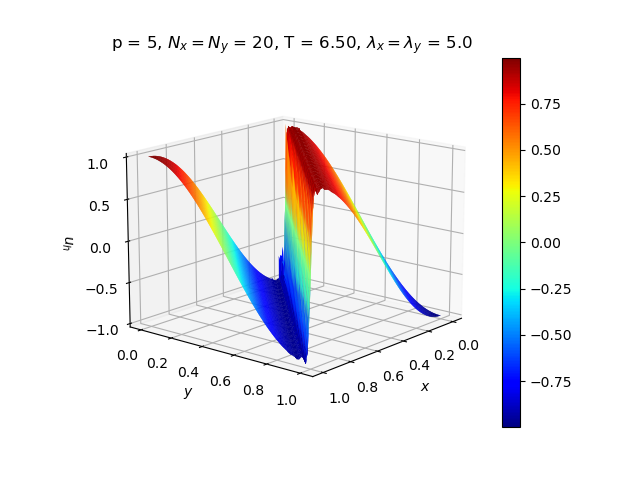}}
    \caption{Discontinuous steady-state problem: DGSEM solutions for a discontinuous steady-state problem $\partial_x u + \partial_y u = 0$ with data $u(x,0)=\cos(\pi x)$, $u(0,y)=-\cos(\pi y)$ obtained with $\lambda_x=\lambda_y=5$, $N_x=N_y=20$, without and with the FCT limiter. The solution is plotted at quadrature points and $T$ refers to the pseudo time required to converge the solution, i.e., $\|u_h^{n+1}-u_h^{n}\|_2 \leq 10^{-14}$.}
    \label{fig:MP_2D_disc_steady_pb}
\end{figure}

\subsubsection{Linear advection-reaction with source}

We finally consider a linear advection-reaction problem with a geometric source term:

\begin{equation}\label{eq:adv_reaction_source_pb_2D}
 \partial_xu + \partial_yu + \beta u = s(x,y) \quad \text{in } \Omega, \quad u(x,0)=u_0(x) \quad \text{in }\Omega.
\end{equation}

\noindent with $\beta\geq0$, $s(\cdot)\geq0$, and nonnegative inflow boundary data. We adapt the problem from \cref{sec:adv_reac_pb_1D} and \cite{Xu_Shu_cons_PP_2023} to two space dimensions and set $\beta=6000$ and a source term $s(x,y)$ such that the solution is $u(x,y)=\tfrac{1}{9}\cos(3\pi x)^4cos(3\pi y)^4$ (see \cref{fig2D_smooth_adv_reaction_pb}). Inflow boundary conditions, $u(x,0)=\tfrac{1}{9}\cos(3\pi x)^4$ and $u(0,y)=\tfrac{1}{9}\cos(3\pi y)^4$, are applied to $x=0$ and $y=0$, while outflow conditions are imposed at $x=1$ and $y=1$.

\begin{figure}
    \centering
    \includegraphics[width=8cm]{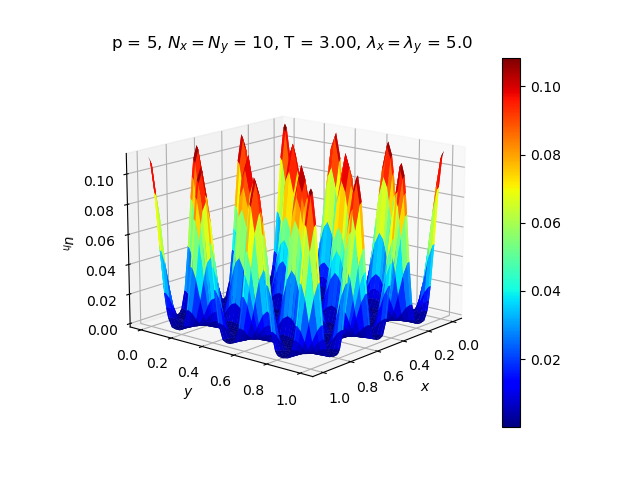}
    \caption{2D advection-reaction with source: steady-state DGSEM solution for problem \cref{eq:adv_reaction_source_pb_2D} with $p=5$, $N_x=N_y=10$ and the FCT limiter. The solution is plotted at quadrature points and $T$ refers to the pseudo time required to converge the solution, i.e., $\|u_h^{n+1}-u_h^{n}\|_2 \leq 10^{-14}$.}
    \label{fig2D_smooth_adv_reaction_pb}
\end{figure}

\Cref{tab:err_2D_smooth_adv_reaction_pb} displays the error levels together with minimum and maximum values of the DOFs obtained without or with the FCT limiter, different approximation orders and different mesh refinements. We again use $\lambda_x=\lambda_y=5$, start from $u_h^{(0)}\equiv0$ and march in time until $\|u_h^{n+1}-u_h^{n}\|_2 \leq 10^{-14}$. As in the 1D case in \cref{sec:adv_reac_pb_1D}, we observe a suboptimal convergence order of $p$ as the mesh is refined due to the insufficient accuracy of the Gauss-Lobatto quadrature rules for integrating the highly nonlinear geometric source terms. Using the limiter or not leads to comparable error levels, while the FCT limiter is necessary for the solution to satisfy the maximum principle.

\begin{table}
    \centering
    \caption{2D advection-reaction with source: $L^{k\in\{2,\infty\}}$ error levels $\|u_h-u\|_{L^k(\Omega_h)}$ and associated orders of convergence ${\cal O}_k$ for problem \cref{eq:adv_reaction_source_pb_2D} with data $u(x,0)=\tfrac{1}{9}\cos(3\pi x)^4$, $u(0,y)=\tfrac{1}{9}\cos(3\pi y)^4$ obtained with $\lambda_x=\lambda_y=5$ when refining the mesh with $N_x=N_y=N$. The solution should remain in the interval $[0,\tfrac{1}{9}]$. Minimum and maximum values of the DOFs: $u_{h_{\min/\max}}=\min/\max(U_{1\leq i,j\leq N}^{0\leq k,l\leq p})$.}
\rotatebox{90}{
    \begin{tabular}{lllllllllllllll}
        \hline
         & & \multicolumn{6}{l}{no limiter} && \multicolumn{6}{l}{with FCT limiter} \\
         \cline{3-8} \cline{10-15}
         $p$ & $N$ & $u_{h_{\min}}$ & $u_{h_{\max}}$ & $L^{2}$ error & ${\cal O}_2$ & $L^{\infty}$ error & ${\cal O}_\infty$ && $u_{h_{\min}}$ & $u_{h_{\max}}$ & $L^{2}$ error & ${\cal O}_2$ & $L^{\infty}$ error & ${\cal O}_\infty$ \\
        \hline
        \multirow{4}{*}{1}
         & 5  &  7.395e-06 & 0.1113 & 1.210E-04 & -- & 2.953E-04 & -- && 7.395e-06 & 1.1111 & 1.212E-04 & -- & 2.953E-04 & --\\
         & 10 & -7.913e-05 & 0.1114 & 4.220E-04 & $-1.80$ & 4.220E-04 & $-0.51$ && 0.0000 & 1.1111 & 4.220E-04 & $-1.80$ & 4.220E-04 & $-0.52$\\
         & 20 & -1.568e-05 & 0.1114 & 5.693E-05 & $2.89$ & 2.933E-04 & $0.52$ && 0.0000 & 1.1111 & 5.634E-05 & $2.91$ & 2.933E-04 & $0.52$\\
         & 40 & -2.206e-06 & 0.1113 & 2.958E-05 & $0.95$ & 1.581E-04 & $0.89$ && 0.0000 & 1.1111 & 1.581E-04 & $0.94$ & 1.581E-04 & $0.89$\\
        \\
        \multirow{4}{*}{2}
         & 5  & -9.775e-08 & 0.1116 & 8.671E-05 & -- & 4.666E-04 & -- && 0.0000 & 1.1111 & 8.539E-05 & -- & 4.666E-04 & --\\
         & 10 & -4.737e-05 & 0.1113 & 3.532E-05 & $1.30$ & 2.196E-04 & $1.09$ && 0.0000 & 1.1111 & 4.664E-05 & $0.87$ & 4.471E-04 & $0.06$\\
         & 20 & -1.107e-05 & 0.1112 & 1.014E-05 & $1.80$ & 5.996E-05 & $1.87$ && 0.0000 & 1.1111 & 1.132E-05 & $2.04$ & 9.358E-05 & $2.26$\\
         & 40 & -1.557e-06 & 0.1111 & 2.630E-06 & $1.95$ & 1.475E-05 & $2.02$ && 0.0000 & 1.1111 & 2.671E-06 & $2.08$ & 1.759E-05 & $2.41$\\
        \\
        \multirow{4}{*}{3}
         & 5  &  2.192e-07 & 0.1114 & 5.032E-05 & -- & 2.604E-04 & -- && 2.192e-07 & 1.1111 & 5.215E-05 & -- & 2.596E-04 & --\\
         & 10 & -8.469e-06 & 0.1111 & 1.029E-05 & $2.29$ & 6.681E-05 & $1.96$ && 0.0000 & 1.1111 & 1.217E-05 & $2.10$ & 1.248E-04 & $1.06$\\
         & 20 & -2.173e-06 & 0.1111 & 1.440E-06 & $2.84$ & 1.091E-05 & $2.61$ && 0.0000 & 1.1111 & 1.722E-06 & $2.82$ & 1.443E-05 & $1.11$\\
         & 40 & -3.801e-07 & 0.1111 & 1.817E-07 & $2.99$ & 1.399E-06 & $2.96$ && 0.0000 & 1.1111 & 2.023E-07 & $3.09$ & 1.693E-06 & $3.09$\\
        \\
        \multirow{4}{*}{4}
         & 5  & -2.96521e-05 & 0.1112 & 2.500E-05 & -- & 9.753E-05 & -- && 0.0000 & 1.1111 & 3.727E-05 & -- & 3.127E-04 & --\\
         & 10 & -7.80861e-06 & 0.1111 & 2.375E-06 & $3.40$ & 1.906E-05 & $2.36$ && 0.0000 & 1.1111 & 4.595E-06 & $3.02$ & 6.962E-05 & $2.17$\\
         & 20 & -3.91095e-07 & 0.1111 & 1.648E-07 & $3.85$ & 1.276E-06 & $3.90$ && 0.0000 & 1.1111 & 1.999E-07 & $4.52$ & 2.491E-06 & $4.80$\\
         & 40 & -1.30899e-08 & 0.1111 & 1.060E-08 & $3.96$ & 8.080E-08 & $3.98$ && 0.0000 & 1.1111 & 1.086E-08 & $4.20$ & 1.009E-07 & $4.63$\\
        \\
        \multirow{4}{*}{5}
         & 5  & -1.13065e-06 & 0.1112 & 1.000E-05 & -- & 7.942E-05 & -- && 0.0000 & 1.1111 & 1.071E-05 & -- & 9.570E-05 & --\\
         & 10 & -5.10247e-07 & 0.1111 & 4.557E-07 & $4.45$ & 3.288E-06 & $4.59$ && 0.0000 & 1.1111 & 5.795E-07 & $4.20$ & 6.230E-06 & $3.94$\\
         & 20 & -3.81951e-08 & 0.1111 & 1.538E-08 & $4.89$ & 1.301E-07 & $4.66$ && 0.0000 & 1.1111 & 2.017E-08 & $4.85$ & 1.740E-07 & $5.16$\\
         & 40 & -1.49843e-09 & 0.1111 & 4.767E-10 & $5.01$ & 3.933E-09 & $5.05$ && 0.0000 & 1.1111 & 5.647E-10 & $5.16$ & 5.807E-09 & $4.91$\\
        \hline
    \end{tabular}
}
    \label{tab:err_2D_smooth_adv_reaction_pb}
\end{table}

%
%
\section{Concluding remarks}\label{sec:conclusions}

This work provides an analysis of the high-order DGSEM discretization with implicit backward-Euler time stepping for the approximation of hyperbolic linear scalar conservation equations in multiple space dimensions. Two main aspects are considered here. We first investigate the maximum principle preservation of the scheme. For the 1D scheme, we prove that the DGSEM preserves the maximum principle of the cell-averaged solution providing that the CFL number is larger than a lower bound. This result allows to use linear scaling limiters \cite{zhang_shu_10a,zhang2012maximum} to impose all the DOFs to satisfy the maximum principle. This property however does not hold in general in multiple space dimensions and we propose to use the FCT limiter \cite{Guermond_IDP_NS_2021,BORIS_Book_FCT_73,zalesak1979fully} to enforce the maximum principle on the cell-averaged solution. The FCT limiter combines the DGSEM scheme with a low-order maximum-principle preserving scheme derived by adding graph viscosity to the DGSEM scheme. The linear scaling limiter is then used to impose the maximum principle to all the DOFs. Numerical experiments in one and two space dimensions are provided to illustrate the conclusions of the present analyses. Then, we investigate the inversion of the linear systems resulting from the time implicit discretization at each time step. We prove that the diagonal blocks are invertible and provide efficient algorithms for their inversion. Future work will concern the extension of this analysis to nonlinear hyperbolic scalar equations and systems of conservation laws on  unstructured grids. Another direction of research may consist in using the fast inversion algorithms introduced in this work for solving preconditionning steps based on tensor product of 1D building blocks in block-preconditionned iterative solvers.


%
%
\appendix

%
%
\section{Multidimensional discrete difference matrix}\label{app:2D_difference_matrix}

The 2D discrete difference matrix reads
\begin{equation*}
{\bf D}_{2d}^\top = \lambda_{x_i}{\bf I}\otimes{\bf D}^\top + \lambda_{y_j}{\bf D}^\top\otimes{\bf I},
\end{equation*}

\noindent and is also nilpotent:
\begin{equation*}
 {\bf D}_{2d}^{2p+1} = 0.
\end{equation*}

Indeed, using properties \cref{eq:prop_Kronecker_prod} we have:
\begin{align*}
 {\bf D}_{2d}^{2p+1} & = (\lambda_{x_i}{\bf I}\otimes{\bf D} + \lambda_{y_j}{\bf D}\otimes{\bf I})^{2p+1} \\
 & = \sum \limits_{k=0}^{2p+1} \binom{2p+1}{k} \lambda_{x_i}^k ({\bf I}\otimes{\bf D})^k \lambda_{y_j}^{2p+1-k}({\bf D}\otimes{\bf I})^{2p+1-k} \\
 & = \sum \limits_{k=0}^{2p+1} \binom{2p+1}{k} \lambda_{x_i}^k \lambda_{y_j}^{2p+1-k}({\bf I}\otimes{\bf D}^k) ({\bf D}^{2p+1-k} \otimes{\bf I}) \\
 & = \sum \limits_{k=0}^{2p+1} \binom{2p+1}{k} \lambda_{x_i}^k \lambda_{y_j}^{2p+1-k} {\bf D}^{2p+1-k} \otimes {\bf D}^k.
\end{align*}

For all $0\leq k \leq 2p+1$, either $2p+1-k$, or $k$ is greater or equal to $p+1$. Hence, nilpotency of ${\bf D}$ \cref{eq:nilpotent_D_matrix} gives the desired result. The matrix is therefore easily invertible:
\begin{equation*}
 \big({\bf I}-{\bf D}_{2d}\big)^{-1} = \sum_{k=0}^{2p} {\bf D}_{2d}^k = \sum_{k=0}^{2p}\sum_{l=0}^k \binom{k}{l}\lambda_{x_i}^{l} \lambda_{y_j}^{k-l} {\bf D}^{l} \otimes{\bf D}^{k-l}.
\end{equation*}

%
%
\section{Inversion of diagonal blocks}\label{app:fast_inversion_diag_blocks}


The linear systems associated to the DGSEM discretization of problem \cref{eq:hyp_cons_laws} with an implicit time stepping have a sparse pattern with dense diagonal blocks of large size. 
One can take advantage of this structure in order to significantly speed up the resolution of such systems with respect to standard inversion algorithms. To support these claims, we implemented the proposed methods and compared them with standard ones. The code is freely available online \cite{FastMatrixInverse_github}. \texttt{python} being our language of choice, we used as reference the linear algebra tools of the popular computational library \texttt{numpy} \cite{numpy}. In what follows, we recall the main equations and operations involved in the inversion of the DGSEM systems.

\subsection{Reminder of main definitions and equations}
At the very core of the DGSEM discretization there is the derivative matrix \eqref{eq:nodalGL_diff_matrix}:
\begin{equation}\label{eq:app_nodalGL_diff_matrix}
 D_{kl} = \ell_l'(\xi_k), \quad 0\leq k,l \leq p,
\end{equation}
where $\ell_l(x)$ are the 1D Lagrange interpolation polynomials \eqref{eq:def_Lag_polynom} and $\xi_k$ the nodes of the Gauss-Lobatto quadrature rule.

Once the time discretization has been taken into account as well, following \cref{sec:DGSEM_1D_fully_discr}, one recover a sparse linear system whose diagonal blocks read ${\bf L}_{1d}{\bf M}$, where, given the quadrature weights $(\omega_k)_{0\leq k\leq p}$,
\begin{equation}
  {\bf M}\coloneqq\tfrac{1}{2}\diag(\omega_0,\dots,\omega_p)
\end{equation}
denotes the 1D mass matrix and ${\bf L}_{1d}$ is given by \eqref{eq:1D_diag_blocks}:
\begin{subequations}
  \begin{align}
    {\bf L}_{1d} \coloneqq{}& {\bf I}-2\lambda\mathbfcal{L},\label{eq:app_1D_diag_block_L}\\
    \mathbfcal{L} \coloneqq{}& {\bf D}^\top-\frac{1}{\omega_p}{\bf e}_p{\bf e}_p^\top.\label{eq:app_1D_diag_block_Lcal}
  \end{align}
\end{subequations}

As discussed in \cref{sec:1D_building_blocks}, the matrix $\mathbfcal{L}$ can be diagonalized \eqref{eq:calL_from_psib}:
\begin{equation}\label{eq:app_calL_eigen_pb}
  \mathbfcal{L} = {\bf R}\Psib{\bf R}^{-1},
\end{equation}
and an explicit formula for the eigenpairs is available \eqref{eq:eigenmodes_matA}:
\begin{subequations}\label{eq:app_calL_diag}
  \begin{align}
    \omega_p\psi^{p+1}+\sum_{l=0}^p\psi^{p-l}D_{pp}^{(l)}={}&0,\\
    r_k ={}& -\frac{1}{\omega_p}\sum_{l=0}^p\psi^{-l-1}D_{pk}^{(l)} \quad \forall 0\leq k\leq p-1, \quad r_p=1.
  \end{align}
\end{subequations}

The system coming from the high-order 2D discretization has diagonal blocks which are easily inverted from the quantities discussed just above. Indeed, system
\begin{equation}\label{eq:app_2D}
  {\bf L}_{2d}({\bf M}\otimes{\bf M}){\bf x}={\bf b},
\end{equation}
with (see \eqref{eq:L2d_mat_def}):
\begin{subequations}\label{eq:app_L2d_mat_def}
\begin{align}
 {\bf L}_{2d} \coloneqq& \frac{\lambda_{x_i}}{\lambda}{\bf I}\otimes{\bf L}_{1d} + \frac{\lambda_{y_j}}{\lambda}{\bf L}_{1d}\otimes{\bf I} = ({\bf R}\otimes{\bf R})\Psib_{2d}({\bf R}\otimes{\bf R})^{-1}, \\
 \Psib_{2d} =& \frac{\lambda_{x_i}}{\lambda}{\bf I}\otimes\Psib_\lambda + \frac{\lambda_{y_j}}{\lambda}\Psib_\lambda\otimes{\bf I}, 
\end{align}
\end{subequations}
can be solved following \eqref{eq:inverse_L2d}:
\begin{subequations}\label{eq:app_2D_inv}
  \begin{align}
   ({\bf M}\otimes{\bf M})^{-1}{\bf L}_{2d}^{-1} &= \left(({\bf M}^{-1}{\bf R})\otimes({\bf M}^{-1}{\bf R})\right)\Psib_{2d}^{-1}\left({\bf R}\otimes{\bf R}\right)^{-1},\\
   \Psib_{2d}^{-1} &= \diag\left(\frac{1}{1-2(\lambda_{x_i}\psi_k+\lambda_{y_j}\psi_l)}:\;1\leq n_{kl}=1+k+l(p+1)\leq N_p\right).
  \end{align}
\end{subequations}

Whenever the graph viscosity is considered in the 2D problem, the diagonal blocks are modified and their closed form reads (see \cref{eq:2D_diag_block_with_graph_visc,eq:matL2d0_U_V}):
\begin{subequations}
  \begin{align}
    {\bf L}_{2d}^v ={}& {\bf L}_{2d}^0 - {\bf U}_v{\bf V}_v^\top,\\
    {\bf L}_{2d}^0 ={}& {\bf L}_{2d} + 2d_{ij}\lambda{\bf I}\otimes{\bf I} = ({\bf R}\otimes{\bf R})\big(\Psib_{2d}+2 d_{ij}\lambda{\bf I}\otimes {\bf I}\big)({\bf R}\otimes{\bf R})^{-1},\\
    {\bf U}_v ={}& 2d_{ij}\big(\lambda_{x_i}{\bf I}\otimes\omegab,\lambda_{y_j}\omegab\otimes{\bf I}\big),\\
    {\bf V}_v ={}& \big({\bf I}\otimes{\bf 1},{\bf 1}\otimes{\bf I}\big).
  \end{align}
\end{subequations}

Hinging on this and applying the Woodbury identity, an efficient way to solve
\begin{equation}\label{eq:app_2D_v}
  {\bf L}_{2d}^v({\bf M}\otimes{\bf M}){\bf x}={\bf b}
\end{equation}
has been proposed in \cref{algo:invert_L2d^v}:
\begin{subequations}\label{eq:app_2D_inv_v}
  \begin{enumerate}
    \item Solve ${\bf L}_{2d}^0{\bf y}={\bf b}$, which gives:
      \begin{equation}
        {\bf y} = ({\bf R}\otimes{\bf R}) \diag\left( \frac{1}{1+2\lambda d_{ij}-2(\lambda_{x_i}\psi_k+\lambda_{y_j}\psi_l)}: \; 1\leq n_{kl}=1+k+l(p+1)\leq N_p \right)({\bf R}^{-1}\otimes{\bf R}^{-1}){\bf b};
      \end{equation}
    \item Solve ${\bf L}_{2d}^0{\bf Z}={\bf U}_v$, which gives:
      \begin{equation}
       {\bf Z} = 2d_{ij}({\bf R}\otimes{\bf R}) \diag\left( \frac{1}{1+2\lambda d_{ij}-2(\lambda_{x_i}\psi_k+\lambda_{y_j}\psi_l)}: 1\leq n_{kl}\leq N_p \right) \big(\lambda_{x_i}{\bf R}^{-1}\otimes({\bf R}^{-1}\omegab),\lambda_{y_j}({\bf R}^{-1}\omegab)\otimes{\bf R}^{-1}\big);
      \end{equation}
    \item Solve
      \begin{equation}
        ({\bf I}_{2p+2}-{\bf V}_v^\top{\bf Z}){\bf z}={\bf V}_v^\top{\bf y}; 
      \end{equation}
    \item Finally, set
      \begin{equation}
        {\bf x} = ({\bf M}^{-1}\otimes{\bf M}^{-1})({\bf y} + {\bf Z}{\bf z}).
      \end{equation}
  \end{enumerate}
\end{subequations}

It is important to notice that matrices $\mathbfcal{L}$, ${\bf R}$, ${\bf M}$, $\Psib$, $\Psib_{2d}$, and related matrices (e.g., ${\bf R}^{-1}$, ${\bf M}^{-1}{\bf R}$,\ldots) depend only on the approximation order of the scheme $p$, not on the $\lambda_{x_i}$ and $\lambda_{y_j}$, so that they may be computed only once at the beginning of the computation.

\subsection{Remarks about the GitHub repository}
Repository \cite{FastMatrixInverse_github} contains a \texttt{python} module, \verb|fast_DGSEM_block_inversion.py| which implements \crefrange{eq:app_nodalGL_diff_matrix}{eq:app_2D_inv_v} and, more importantly, assess \eqref{eq:app_2D_inv} (respectively, \eqref{eq:app_2D_inv_v} and \eqref{eq:app_calL_diag}) by comparing it in terms of exactness to the result of \eqref{eq:app_2D} (resp., \eqref{eq:app_2D_v} and \eqref{eq:app_calL_eigen_pb}) obtained with reference algebraic tools (mainly, \texttt{numpy.linalg.inv}).

Several other optimizations, especially regarding operations involving diagonal matrices have been considered, and, for the sake of fairness, used both in the proposed and standard ways of solving \cref{eq:app_calL_eigen_pb,eq:app_2D,eq:app_2D_v}.

The performances of such resolution methods can be computed and (visually) analyzed thanks to the notebook \verb|test_fast_dgsem.ipynb|. Indeed, we give in \cref{fig:app_perf} the results obtained with this notebook on a personal machine with 8 Intel Xeon(R) W-2223 CPUs and 16Gb RAM. Performance analysis has been evaluated thanks to built-in \texttt{python} module \texttt{timeit}; statistical data has been computed over 20 runs, each calling the procedure under evaluation more than 1000 times. Even though such performance measures might vary from one machine to the other, and, also, from one run to the other, we can reliably say that inversion strategies \eqref{eq:app_2D_inv} and \eqref{eq:app_2D_inv_v} show consistent and often significant performance gains with respect to their dense counterparts. Admittedly, the gains are less noticeable for \eqref{eq:app_2D_v} whenever high orders are used, see right part of \cref{fig:app_perf} (indeed, simple computations show that \eqref{eq:app_2D_inv_v}, more precisely, the matrix products involved there, and the dense resolution of \eqref{eq:app_2D_v} have similar algorithmic complexity). Nonetheless, since the graph viscosity is not always necessary (see \cref{sec:FCT_limiter}), the overall performances of the two-staged FCT limiter highly benefits from the proposed inversion strategies. All in all, to tackle the problem with graph viscosity \eqref{eq:app_2D_v}, we advise to prefer procedure \eqref{eq:app_2D_inv_v} over a dense solve whenever the polynomial order is moderate.

\begin{figure}
    \centering
    \subfigure[Resolution of \eqref{eq:app_2D}]{
      \includegraphics[height=6cm]{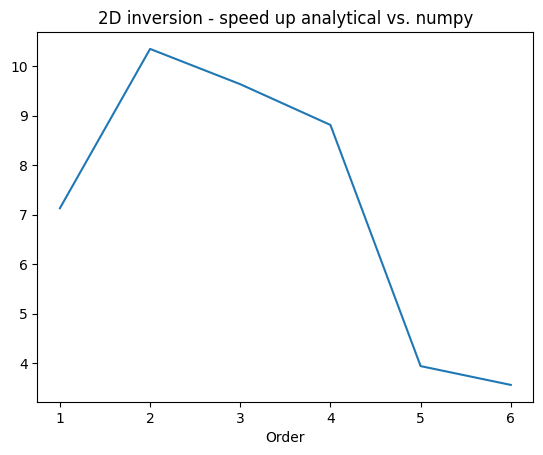}
    }\hfill
    \subfigure[Resolution of \eqref{eq:app_2D_v}]{
      \includegraphics[height=6cm]{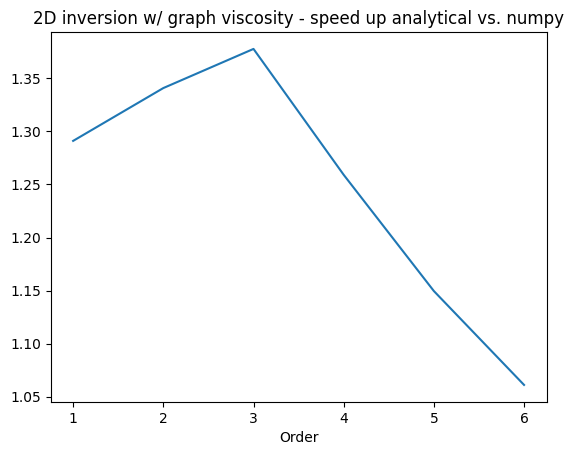}
    }
    \caption{Left: performance speed-up over polynomial degree for solving \eqref{eq:app_2D} with proposed procedure \eqref{eq:app_2D_inv} with respect to standard algebraic tools (\texttt{numpy}). Right: similar to above, but for \eqref{eq:app_2D_v} and  \eqref{eq:app_2D_inv_v}.}
    \label{fig:app_perf}
\end{figure}

Finally, the testing framework \verb|test_fast_dgsem.py| allows one to check in a compact way the exactness of proposed formulae \cref{eq:app_calL_diag,eq:app_2D_inv,eq:app_2D_inv_v} and other matrix-related optimizations for several settings at once.

%
%
\section{The 3D DGSEM scheme}\label{app:3D_DGSEM_scheme}

We here give details and the main properties of the time implicit DGSEM scheme for the approximation of \cref{eq:hyp_cons_laws} with flux ${\bf f}({\bf u})=u(c_x,c_y,c_z)^\top$ and nonnegative $c_x,c_y$, and $c_z$. We consider a Cartesian mesh with elements with measure $|\kappa_{ijk}|=\Delta x_i\times \Delta y_j\times \Delta z_k$ and set $\lambda_{x_i}=\tfrac{c_x\Delta t}{\Delta x_i}$, $\lambda_{y_j}=\tfrac{c_y\Delta t}{\Delta y_j}$, $\lambda_{z_k}=\tfrac{c_z\Delta t}{\Delta z_k}$.

\subsection{High-order and low-order schemes}

Using a vector storage of the DOFs as $({\bf U}_{ijk})_{n_{lmr}}=U_{ijk}^{lmr}$ with $1\leq n_{lmr}\coloneqq1+l+m(p+1)+r(p+1)^2\leq N_p$ and $N_p=(p+1)^3$, the discrete scheme with graph viscosity under vector form reads

\begin{equation}\label{eq:3D_discr_DGSEM_lin_vector_form}
\begin{aligned}
 ({\bf M}\otimes{\bf M}\otimes{\bf M})({\bf U}_{ijk}^{n+1} -{\bf U}_{ijk}^{n}) &- \lambda_{x_i}\Big({\bf M}\otimes{\bf M}\otimes(2{\bf D}^\top{\bf M}-{\bf e}_{p}{\bf e}_{p}^\top)\Big){\bf U}_{ijk}^{n+1} - \lambda_{x_i}\big({\bf M}\otimes{\bf M}\otimes{\bf e}_{0}{\bf e}_{p}^\top\big){\bf U}_{(i-1)jk}^{n+1} \\
  &- \lambda_{y_j}\Big({\bf M}\otimes(2{\bf D}^\top{\bf M}-{\bf e}_{p}{\bf e}_{p}^\top)\otimes{\bf M}\Big){\bf U}_{ijk}^{n+1} - \lambda_{y_j}\big({\bf M}\otimes{\bf e}_{0}{\bf e}_{p}^\top\otimes{\bf M}\big){\bf U}_{i(j-1)k}^{n+1} \\
  &- \lambda_{z_k}\Big((2{\bf D}^\top{\bf M}-{\bf e}_{p}{\bf e}_{p}^\top)\otimes{\bf M}\otimes{\bf M}\Big){\bf U}_{ijk}^{n+1} - \lambda_{z_k}\big({\bf e}_{0}{\bf e}_{p}^\top\otimes{\bf M}\otimes{\bf M}\big){\bf U}_{ij(k-1)}^{n+1} + {\bf V}_{ijk}^{(n+1)} = 0,
\end{aligned}
\end{equation}

\noindent where ${\bf M}=\tfrac{1}{2}\text{diag}(\omega_0,\dots,\omega_p)$ denotes the 1D mass matrix, $({\bf e}_k)_{0\leq k\leq p}$ is the canonical basis of $\mathbb{R}^{p+1}$. The graph viscosity is defined by

\begin{align}\label{eq:3D_graph_viscosity}
 {\bf V}_{ijk}^{(n+1)} &= 2d_{ijk}\Big(\lambda_{x_i}{\bf M}\otimes{\bf M}\otimes\big({\bf M}-\omegab{\bf 1}^\top{\bf M}\big)+\lambda_{y_j}{\bf M}\otimes\big({\bf M}-\omegab{\bf 1}^\top{\bf M}\big)\otimes{\bf M}+\lambda_{z_k}\big({\bf M}-\omegab{\bf 1}^\top{\bf M}\big)\otimes{\bf M}\otimes{\bf M}\Big){\bf U}_{ij}^{(n+1)} 
\end{align}

\noindent with $d_{ijk}\geq0$, $\omegab=\tfrac{1}{2}(\omega_0,\dots,\omega_p)^\top$ and ${\bf 1}=(1,\dots,1)^\top\in\mathbb{R}^{p+1}$.  Scheme \cref{eq:3D_discr_DGSEM_lin_vector_form} with $d_{ijk}=0$ (resp.,  $d_{ijk}>0$) denotes the high-order (resp., low-order) scheme. The linear system \cref{eq:3D_discr_DGSEM_lin_vector_form} is maximum principle preserving under the same condition \cref{eq:cond_on_dij_for_M_matrix} as in 2D: $d_{ijk}\geq 2\max_{0\leq k \neq m \leq p}\big(-\tfrac{D_{mk}}{\omega_k}\big)$.

The 3D discrete difference matrix  reads

\begin{equation*}
{\bf D}_{3d}^\top = \lambda_{x_i}{\bf I}\otimes{\bf I}\otimes{\bf D}^\top + \lambda_{y_j}{\bf I}\otimes{\bf D}^\top\otimes{\bf I} + \lambda_{z_k}{\bf D}^\top\otimes{\bf I}\otimes{\bf I},
\end{equation*}

\noindent and is also nilpotent: ${\bf D}_{3d}^{3p+1} = 0$.

\subsection{FCT limiter}

By $u_{h,HO}^{(n+1)}$ we denote the high-order solution to \cref{eq:3D_discr_DGSEM_lin_vector_form} with $d_{ijk}=0$ and by $u_{h,LO}^{(n+1)}$ we denote the high-order solution to \cref{eq:3D_discr_DGSEM_lin_vector_form} with $d_{ijk} = 2\max_{0\leq k \neq m \leq p}\big(-\tfrac{D_{mk}}{\omega_k}\big)$. Applying the FCT limiter introduced in \cref{sec:FCT_limiter}, the limited DOFs are evaluated explicitly from $u_{h,LO}^{(n+1)}$ and  $u_{h,HO}^{(n+1)}$ through

\begin{align*}
        U_{ijk}^{lmr,n+1} - U_{ijk,HO}^{lmr,n+1} &= \delta_{lp} \frac{2\lambda_{x_i}}{\omega_p}\big(1-l_{ijk}^{(i+1)jk}\big)\Big(U_{ijk,HO}^{pmr,n+1}-U_{ijk,LO}^{pmr,n+1}\Big)  - \delta_{l0} \frac{2\lambda_{x_i}}{\omega_0}\big(1-l_{ijk}^{(i-1)jk}\big)\Big(U_{(i-1)jk,HO}^{pmr,n+1}-U_{(i-1)jk,LO}^{pmr,n+1}\Big) \\
    & + \delta_{mp} \frac{2\lambda_{y_j}}{\omega_p}\big(1-l_{ijk}^{i(j+1)k}\big)\Big(U_{ijk,HO}^{lpr,n+1}-U_{ijk,LO}^{lpr,n+1}\Big) - \delta_{m0} \frac{2\lambda_{y_j}}{\omega_0}\big(1-l_{ijk}^{i(j-1)k}\big)\Big(U_{i(j-1)k,HO}^{lpr,n+1}-U_{i(j-1)k,LO}^{lpr,n+1}\Big) \\
    & + \delta_{rp} \frac{2\lambda_{z_k}}{\omega_p}\big(1-l_{ijk}^{ij(k+1)}\big)\Big(U_{ijk,HO}^{lmp,n+1}-U_{ijk,LO}^{lmp,n+1}\Big) - \delta_{r0} \frac{2\lambda_{z_k}}{\omega_0}\big(1-l_{ijk}^{ij(k-1)}\big)\Big(U_{ij(k-1),HO}^{lmp,n+1}-U_{ij(k-1),LO}^{lmp,n+1}\Big),
\end{align*}

\noindent where

\begin{equation*}
 \begin{array}{rcl}
    l_{ijk}^{lmr} &=& \left\{
    \begin{array}{ll}
        \min(l_{ijk}^{-},l_{lmr}^{+}) & \mbox{if } A_{ijk}^{lmr} < 0, \\
        \min(l_{lmr}^{-},l_{ijk}^{+}) & \mbox{else,}
    \end{array}\right. \\
 l_{ijk}^{\pm} &=& \min\Big(1,\frac{Q_{ijk}^{\pm}}{P_{ijk}^{\pm}}\Big),
 \end{array} \quad
 \begin{array}{ll}
   A_{ijk}^{(i+1)jk} &= -A_{(i+1)jk}^{ijk} = \displaystyle\sum_{0\leq m,r\leq p}\frac{\omega_m\omega_r}{4}\big(U_{ijk,HO}^{pmr,n+1}-U_{ijk,LO}^{pmr,n+1}\big), \\
    A_{ijk}^{i(j+1)k} &= -A_{i(j+1)k}^{ijk} = \displaystyle\sum_{0\leq l,r\leq p}\frac{\omega_l\omega_r}{4}\big(U_{ijk,HO}^{lpr,n+1}-U_{ijk,LO}^{lpr,n+1}\big), \\
    A_{ijk}^{ij(k+1)} &= -A_{ij(k+1)}^{ijk} = \displaystyle\sum_{0\leq m,l\leq p}\frac{\omega_m\omega_l}{4}\big(U_{ijk,HO}^{lmp,n+1}-U_{ijk,LO}^{lmp,n+1}\big),
 \end{array}
\end{equation*}


\begin{equation*}
P_{ijk}^{-} = \sum_{(l,m,r)\in \mathcal{S}(i,j,k)} \min\big(A_{ijk}^{lmr},0\big),  \quad Q_{ijk}^{-} = m - \avg{u_{LO}^{(n+1)}}_{ijk}, \quad
P_{ijk}^{+} = \sum_{(l,m,r)\in \mathcal{S}(i,j,k)} \max\big(A_{ijk}^{lmr},0\big), \quad Q_{ijk}^{+} = M - \avg{u_{LO}^{(n+1)}}_{ijk} \geq 0,
\end{equation*}

\noindent and  $\mathcal{S}(i,j,k) = \{(i\pm1,j,k);(i,j\pm1,k);(i,j,k\pm1)\}$. The limited solution is bounded by the lower and upper bounds in \cref{eq:PDE_max_principle}: $m\leq \avg{{u}_{h}^{(n+1)}}_{ijk}:=\tfrac{1}{8}\sum_{lmr}\omega_l\omega_m\omega_rU_{ijk}^{lmr,n+1}\leq M$ and the limiter keeps conservation of the scheme:

\begin{equation*}
 \sum_{ijk}\avg{{u}_{h}^{(n+1)}}_{ijk} = \sum_{ijk}\avg{{u}_{h,HO}^{(n+1)}}_{ijk} = \sum_{ijk}\avg{{u}_{h,LO}^{(n+1)}}_{ijk} = \sum_{ijk}\avg{{u}_{h}^{(n)}}_{ijk}
\end{equation*}

\noindent for periodic boundary conditions or compactly supported solutions.

\subsection{Inversion of diagonal blocks}

The global system to be solved at each time step reads

\begin{equation*}
 \mathbb{A}_{3d}{\bf U}^{(n+1)} = \mathbb{M}_{3d}{\bf U}^{(n)},
\end{equation*}

\noindent of size $N_xN_yN_zN_p$ with blocks of size $N_p = (p + 1)^3$. The diagonal blocks without graph viscosity read ${\bf L}_{3d}({\bf M}\otimes{\bf M}\otimes{\bf M})$ with

\begin{equation*}
 {\bf L}_{3d} = ({\bf R}\otimes{\bf R}\otimes{\bf R})\Psib_{3d}({\bf R}\otimes{\bf R}\otimes{\bf R})^{-1}, \quad \Psib_{3d} = \frac{\lambda_{x_i}}{\lambda}{\bf I}\otimes{\bf I}\otimes\Psib_\lambda + \frac{\lambda_{y_j}}{\lambda}{\bf I}\otimes\Psib_\lambda\otimes{\bf I} + \frac{\lambda_{z_k}}{\lambda}\Psib_\lambda\otimes{\bf I}\otimes{\bf I},
\end{equation*}

\noindent where $\lambda\coloneqq\lambda_{x_i}+\lambda_{y_j}+\lambda_{z_k}$, and where ${\bf R}$ and $\Psib_\lambda$ are defined in \cref{eq:L1D_from_psib}. Note that $\Psib_{3d}$ is a diagonal matrix, so the diagonal blocks are thus easily inverted:

\begin{equation*}
 \big({\bf L}_{3d}({\bf M}\otimes{\bf M}\otimes{\bf M})\big)^{-1} = ({\bf M}^{-1}{\bf R}\otimes{\bf M}^{-1}{\bf R}\otimes{\bf M}^{-1}{\bf R})\Psib_{3d}^{-1}({\bf R}^{-1}\otimes{\bf R}^{-1}\otimes{\bf R}^{-1}).
\end{equation*}

Including graph viscosity, the diagonal blocks are defined by ${\bf L}_{3d}^v({\bf M}\otimes{\bf M}\otimes{\bf M})$ with

\begin{equation*}
 {\bf L}_{3d}^v =  {\bf L}_{3d}^0 - {\bf U}_v{\bf V}_v^\top, \quad {\bf L}_{3d}^0 = ({\bf R}\otimes{\bf R}\otimes{\bf R})\big(\Psib_{3d}+2d_{ijk}\lambda{\bf I}\otimes{\bf I}\otimes{\bf I}\big)({\bf R}\otimes{\bf R}\otimes{\bf R})^{-1},
\end{equation*}

\noindent with

\begin{equation*}
 {\bf U}_v = 2d_{ijk}\big(\lambda_{x_i}{\bf I}\otimes{\bf I}\otimes\omegab,\lambda_{y_j}{\bf I}\otimes\omegab\otimes{\bf I},\lambda_{z_k}\omegab\otimes{\bf I}\otimes{\bf I}\big), \quad {\bf V}_v = \big({\bf I}\otimes{\bf I}\otimes{\bf 1},{\bf I}\otimes{\bf 1}\otimes{\bf I},{\bf 1}\otimes{\bf I}\otimes{\bf I}\big),
\end{equation*}

\noindent in $\mathbb{R}^{N_p\times3(p+1)^2}$. Once again, the diagonal blocks with graph viscosity may be inverted with \cref{algo:invert_L3d^v} using the Woodbury identity.

\begin{algorithm}
\caption{Algorithm flowchart for solving the system ${\bf L}_{3d}^v({\bf M}\otimes{\bf M}\otimes{\bf M}){\bf x}={\bf b}$ with graph viscosity.}\label{algo:invert_L3d^v}
\begin{algorithmic}[1]
\STATE{solve ${\bf L}_{3d}^0{\bf y}={\bf b}$ for ${\bf y}\in\mathbb{R}^{N_p}$:}

\begin{equation*}
 {\bf y} = ({\bf R}\otimes{\bf R}\otimes{\bf R}) \diag\left( \frac{1}{1+2\lambda d_{ijk}-2(\lambda_{x_i}\psi_l+\lambda_{y_j}\psi_m+\lambda_{z_k}\psi_r)}, \; 0\leq l,m,r\leq p \right)\left({\bf R}^{-1}\otimes{\bf R}^{-1}\otimes{\bf R}^{-1}\right){\bf b};
\end{equation*}

\STATE{solve ${\bf L}_{3d}^0{\bf Z}={\bf U}_v$ for ${\bf Z}\in\mathbb{R}^{N_p\times3(p+1)^2}$:}

\begin{multline*}
 {\bf Z} = 2d_{ijk}\left({\bf R}\otimes{\bf R}\otimes{\bf R}\right) \diag\left( \frac{1}{1+2\lambda d_{ijk}-2(\lambda_{x_i}\psi_l+\lambda_{y_j}\psi_m+\lambda_{z_k}\psi_r)}, \; 0\leq l,m,r\leq p \right) \cdots \\ \cdots \left(\lambda_{x_i}{\bf R}^{-1}\otimes{\bf R}^{-1}\otimes({\bf R}^{-1}\omegab),\lambda_{y_j}{\bf R}^{-1}\otimes({\bf R}^{-1}\omegab)\otimes{\bf R}^{-1},\lambda_{z_k}({\bf R}^{-1}\omegab)\otimes{\bf R}^{-1}\otimes{\bf R}^{-1}\right);
\end{multline*}

\STATE{solve $({\bf I}_{3(p+1)^2}-{\bf V}_v^\top{\bf Z}){\bf z}={\bf V}_v^\top{\bf y}$ for ${\bf z}\in\mathbb{R}^{3(p+1)^2}$;}

\STATE{set ${\bf x} = ({\bf M}^{-1}\otimes{\bf M}^{-1}\otimes{\bf M}^{-1})({\bf y} + {\bf Z}{\bf z})$.}
\end{algorithmic}
\end{algorithm}

%
%
\bibliographystyle{siamplain}
\bibliography{biblio_generale}

\end{document}